\PassOptionsToPackage{dvipsnames}{xcolor}
\documentclass[tikz,border=5pt,leqno]{siamart1116}

\usepackage{amssymb}
\usepackage{graphicx}
\usepackage{xspace}
\usepackage{multirow}
\usepackage{siunitx}
\usepackage{dutchcal}
\usepackage{derivative}
\usepackage{upgreek}

\usepackage{bm,bbm}
\usepackage{physics}
\usepackage{./Macros/mydef}
\usepackage[numbers]{natbib}
\let\cite=\citet
\usepackage{enumerate}
\usepackage{dsfont}
\usepackage{enumitem}
\usepackage{booktabs}
\usepackage{algorithm}
\usepackage{algpseudocode}

\usepackage{tikz}
\usetikzlibrary{arrows,decorations,shapes, patterns}

\usepackage{pgfplots}
\pgfplotsset{compat=1.18}

\pgfmathsetmacro{\A}{7.3}
\pgfmathsetmacro{\taunot}{0.1117}
\pgfmathsetmacro{\B}{3.9}
\pgfmathsetmacro{\gnot}{0.5}
\pgfmathsetmacro{\cv}{0.000389}
\pgfmathsetmacro{\Tnot}{150}
\pgfmathdeclarefunction{sie}{2}{%
  \pgfmathparse{\A*\taunot * ((#1/\taunot)^(-\B) + (#1/\taunot)*\B - (\B+1)) + \cv*#2*(1-\Tnot / (\Tnot + #2*(#1/\taunot)^(\gnot)))}
}

\pgfmathsetmacro{\DavA}{2.434}
\pgfmathsetmacro{\DavB}{2.367}
\pgfmathsetmacro{\DavC}{1.024}
\pgfmathsetmacro{\DavV}{0.5848}
\pgfmathsetmacro{\DavG}{0.8437}
\pgfmathsetmacro{\DavT}{298.15}
\pgfmathsetmacro{\DavCV}{0.001072}
\pgfmathsetmacro{\DavST}{0.8484}
\pgfmathdeclarefunction{ref_B}{1}{%
  \pgfmathparse{ 
    (#1 < \DavV)
    ? 
    \DavA / \DavV^2 * ( (4*\DavB * (1 - #1/\DavV))^2  + 4*\DavB*(1 - #1/\DavV) + 1 + \DavC * (4*\DavB*(1 - #1/\DavV))^3/6 + (1 - #1/\DavV)*(2 - #1/\DavV) / (2 * \DavB * (#1/\DavV)^5))
    :
    \DavA / \DavV^2 * exp(4*\DavB*(1-#1/\DavV))
  }
}

\pgfmathdeclarefunction{ref_T}{1}{%
  \pgfmathparse{ 
    \DavT / \DavV^2 * \DavG * (\DavG + 1) * (#1 / \DavV)^(-\DavG - 2)
  }
}

\pgfmathdeclarefunction{isen_c}{2}{%
  \pgfmathparse{ 
    ( (#1) * ref_B(#1) +  \DavCV / (1+\DavST) * ( (\DavST / \DavCV * #2)^(1+1/\DavST) - 1) * (#1)^2 * ref_T(#1)))^(1/2)
  }
}

\pgfmathsetmacro{\pL}{2}
\pgfmathsetmacro{\pR}{5}
\pgfmathsetmacro{\pinfL}{4}
\pgfmathsetmacro{\pinfR}{1}
\pgfmathsetmacro{\rhoL}{1}
\pgfmathsetmacro{\rhoR}{2}
\pgfmathsetmacro{\vL}{1.0}
\pgfmathsetmacro{\vR}{-1.0}

\pgfmathdeclarefunction{phiphi}{1}{%
   \pgfmathparse{
     f_L(#1) + f_R(#1) + \vR - \vL
   }
}
\pgfmathdeclarefunction{f_L}{1}{%
  \pgfmathparse{
    (#1 < \pL) 
    ?
    ((\pL+\pinfL)/\rhoL)^(1/2) * ln((#1+\pinfL)/(\pL+\pinfL))
    :
    (#1 - \pL) / (\rhoL * (#1 + \pinfL))^(1/2)
  }
}
\pgfmathdeclarefunction{f_R}{1}{%
  \pgfmathparse{
    (#1 < \pR) 
    ?
    ((\pR+\pinfR)/\rhoR)^(1/2) * ln((#1+\pinfR)/(\pR+\pinfR))
    :
    (#1 - \pR) / (\rhoR * (#1 + \pinfR))^(1/2)
  }
}
\pgfmathdeclarefunction{phiRR}{1}{%
   \pgfmathparse{
     ((\pL+\pinfL)/\rhoL)^(1/2) * ln((#1+\pinfR)/(\pL+\pinfL)) + ((\pR+\pinfR)/\rhoR)^(1/2) * ln((#1+\pinfR)/(\pR+\pinfR)) + \vR - \vL
   }
}
\pgfmathdeclarefunction{phiSS}{1}{%
   \pgfmathparse{
     (#1-\pL)/(\rhoL*(#1 + \pinfL))^(1/2) + (#1-\pR)/(\rhoR*(#1 + \pinfL))^(1/2) + \vR - \vL
   }
}

\pgfmathsetmacro{\stiffq}{-700}
\pgfmathsetmacro{\stiffp}{1000}
\pgfmathsetmacro{\stiffg}{1.4}
\pgfmathdeclarefunction{stiff_isen}{2}{%
   \pgfmathparse{
     \stiffq + \stiffp * #1 + #2 * (#1)^(-(\stiffg-1.0))
   }
}

\newcommand{\myhide}[1]{{}}

\begin{document}

\newcommand{\TheTitle}{%
  Preserving the minimum principle on the entropy for the compressible Euler Equations with general equations of state}
\newcommand{\TheAuthors}{B. Clayton, E.~J. Tovar}

\headers{Preserving the minimum entropy principle}{\TheAuthors}

\title{{\TheTitle}\thanks{Draft version, \today %
    \funding{
      BC acknowledges the support of the Eulerian Applications Project (within the Advanced Simulation and Computing program) at Los Alamos National Laboratory (LANL).
      ET acknowledges the support by the Laboratory Directed Research and Development (LDRD) program under project number 20240555MFR and the Engineering Technology Maturation (ETM) program of Los Alamos National Laboratory.
      ET also acknowledges the support from the U.S. Department of Energy’s Office of Applied Scientific Computing Research (ASCR) and Center for Nonlinear Studies (CNLS) at LANL under the Mark Kac Postdoctoral Fellowship in Applied Mathematics.
      LANL is operated by Triad National Security, LLC, for the National Nuclear Security Administration of the U.S. Department of Energy (Contract No. 89233218CNA000001). The Los Alamos unlimited release number is LA-UR-25-22227.}}}

\author{Bennett Clayton\footnotemark[2]
  \and Eric J. Tovar\footnotemark[1]}

\maketitle

\renewcommand{\thefootnote}{\fnsymbol{footnote}}

\footnotetext[1]{%
  Theoretical Division, Los Alamos National Laboratory, P.O. Box 1663, Los Alamos, NM, 87545, USA.}

\footnotetext[2]{%
  X Computational Physics Division, Los Alamos National Laboratory, P.O. Box 1663, Los Alamos, NM, 87545, USA.}

\renewcommand{\thefootnote}{\arabic{footnote}}

\begin{abstract}
  This paper is concerned with constructing an invariant-domain preserving approximation technique for the compressible Euler equations with general equations of state that preserves the minimum principle on the physical entropy.
  We derive a sufficient wave speed estimate for the Riemann problem under some mild thermodynamic assumptions on the equation of state.
  This minimum principle is guaranteed through the use of discrete auxiliary states which are in the invariant domain when using this new wave speed estimate.
  Finally, we numerically illustrate the proposed methodology.
\end{abstract}




\section{Introduction}
\label{sec:introduction}
The compressible Euler equations are used in simulating a variety of applications from the interaction of supersonic flow around obstacles to full multi-physics applications such as the detonation of high-explosives.
As more complexity is added to the application of interest, the equation of state (EOS) used to complete the system becomes more involved.
Further issues can arise when the material of interest experiences tension; that is, it exhibits states with negative pressure and potentially negative specific internal energy.
This is often the case when the equation of state is used to describe properties of condensed phases such as solids or liquids (see:~\citep{lozanoSimpleMACAW, lozano2023analytic}).
As positivity of pressure and specific internal energy may be irrelevant depending on the equation of state, the thermodynamically relevant admissible set for the Euler equations requires positivity of material temperature (and density).
It can be shown that preservation of the minimum principle on the specific entropy implies positivity of temperature.
Thus, a numerical method which preserves this minimum principle will yield physically relevant solutions.

It was shown in \citet{tadmor1986minimum} that entropy solutions to the Euler equations (with polytropic EOS) satisfy a local minimum principle on the specific entropy.
Minimum entropy principle preserving numerical methods have been developed for the ideal gas equation of state over the past few decades (see: \citep{khobalatte1994maximum, coquel1998relaxation, frid2001maps, zhang2012minimum, lv2015entropy, guermond2016}).
In general, methods which preserve a discrete entropy inequality for every mathematical entropy, can be shown to also preserve the minimum principle on the specific entropy; an example of this is seen for the Lax-Friedrichs scheme in \cite{kroner2008minimum}.

It may be possible to preserve the minimum principle on the entropy for an analytic equation of state if one solves the Riemann problem exactly in conjunction with the Godunov method.
However, exact Riemann solvers are very computationally expensive.
More details on this process can be found in \cite[Sec. 1]{colella1985efficient}, \cite{quartapelle2003solution}, or \cite{ivings1998riemann}.
Furthermore, certain approximate Riemann solvers, in particular the HLLC method in \citet{toro1994restoration} (among other methods) can be made to preserve the minimum principle on the specific entropy following the methodology in~\cite{berthon2012local}.
As is the case in any approximate Riemann solver, knowledge of some satisfactory wave speed estimate is required and it is not shown in~\citep{berthon2012local} how to find this wave speed estimate for a minimum principle preserving method.
A sufficient estimation of the maximum wave speed is an essential ingredient and is at the heart of these methods.
The premise of this paper is to provide such an estimation of the wave speed.
To the best of our knowledge, an efficient and robust method for provably preserving the minimum principle on the specific entropy for general equations of state has been absent.
While not minimum principle preserving, some recent work addressing arbitrary EOS can be seen in~\citep{sirianni2024explicit, aiello2025entropy, orlando2025asymptotic}.
Two other recent approaches construct ``alternative'' Riemann problems which approximate the original Riemann problem, see~\cite{koroleva2024approximate} and~\citet{wang2023stiffened}.
This latter approach is in somewhat the same vein as the method outlined in this paper.

The starting point for this work is based on~\citep{clayton_2022,guermond2023} wherein an interpolatory pressure is defined and used to estimate a wave speed in a local Riemann problem.
This method was originally inspired by~\cite{abgrall2001computations}.
A limitation of~\citep{clayton_2022, guermond2023} is that pressure and specific internal energy must be positive.
Furthermore, these constraints may not be physically relevant depending on the equation of state.
As previously stated, the admissible set of interest requires only positivity of the density and temperature.
The main contribution of this work is the derivation of a sufficient wave speed estimate in the local Riemann problem under a thermodynamic constraint based on the fundamental derivative.
This constraint is necessary for identifying when the minimum principle on the entropy can be preserved and is used in the proof of Lemma~\ref{lem:expansion_wave}.
This wave speed estimate is found by interpolating the pressure and the sound speed with a simpler equation of state which allows for an analytic derivation.
We then show that one can construct auxiliary states based on this wave speed estimate which preserve the minimum principle on the specific entropy.
Furthermore, we show that any approximation technique that is a convex combination of the discrete analog of the auxiliary states will satisfy the constraints in the admissible set (\ie is invariant-domain preserving).
The work here will lay the foundation for developing a higher-order invariant-domain approximation technique for the Euler equations with general equations of state.

This paper is organized as follows.
In Section~\ref{sec:model}, we introduce the compressible Euler equations, the necessary thermodynamic constraints and the notion of an admissible set.
We then recall the Riemann problem in Section~\ref{sec:riemann_problem}. Then, in Section~\ref{sec:approx_details}, we introduce a generic approximation technique that is a convex combination of some auxiliary states.
The extended Riemann problem is introduced in Section~\ref{sec:extended_riemann}. The main results of the paper are given in Lemma~\ref{lem:invariant_domain} and Theorem~\ref{thm:main_theorem}.
In Section~\ref{sec:eos_details}, we give details for specific equations of state that satisfy the necessary thermodynamic constraints.
We numerically illustrate the proposed methodology in Section~\ref{sec:numerical_results} with several one-dimensional and two-dimensional problems. We finish with a brief conclusion.
\section{The model} \label{sec:model}
In this section, we formulate the model problem and introduce notation. We also give a brief discussion on equations of state, thermodynamics and admissible sets for the compressible Euler equations.
\subsection{The compressible Euler equations}
\label{sec:euler-equations}
Consider a compressible, inviscid fluid occupying a bounded, polyhedral domain $D$ in $\polR^d$ where $d\in\{1,2,3\}$ is the spatial dimension. We assume that no external forces are acting on the fluid and can thus be modeled by the compressible Euler equations. Given some initial data $\bu_0(\bx)\eqq(\rho_0, \bbm_0, E_0)(\bx)$ over $D$ describing the fluid at the initial time, $t_0$, we look for $\bu(\bx,t)\eqq(\rho, \bbm, E)(\bx,t)$ solving the compressible Euler equations (in the weak sense):
\begin{subequations}\label{euler}
    \begin{align}
         & \partial_t \rho + \DIV\bbm = 0                                                  & \text{a.e. }t>t_0,\,\bx\in D, \label{eq:mass_cons}     \\
         & \partial_t \bbm + \DIV\left(\bv \otimes \bbm + \sfp(\bu)\polI_d\right) = \bm{0} & \text{a.e. }t>t_0,\,\bx\in D, \label{eq:momentum_cons} \\
         & \partial_t E + \DIV\left(\bv\big(E + \sfp(\bu)\big) \right) = 0                 & \text{a.e. }t>t_0,\,\bx\in D. \label{eq:energy_cons}
    \end{align}
\end{subequations}
Here, the components of the unknown variable $\bu \eqq (\rho, \bbm, E)\T$ are the density, $\rho$, the momentum vector, $\bbm$, and the total mechanical energy, $E$.
Note that $\bu$ is considered to be a column vector.
The velocity vector is defined by $\bv(\bu)\eqq\rho^{-1}\bbm$.
We define the specific volume of the fluid as $\tau(\bu) \eqq\rho^{-1}$. For the sake of brevity, we will drop the dependence of $\bu$ when talking about the specific volume.
We also introduce the short-hand notion for the flux of the system: $\polR^{d+2} \ni \bu \mapsto \polf(\bu)\eqq (\bbm, \bv(\bu)\otimes\bbm + \sfp(\bu)\polI_d,\bv(\bu)(E + \sfp(\bu)))\T \in \polR^{(d + 2)\times d}$ where $\polI_d$ is the $d\times d$ identity matrix.
We also introduce the specific internal energy of the fluid $\sfe(\bu)\eqq\rho^{-1}E - \tfrac12\|\bv(\bu)\|^2_{\ell^2}$; the quantity $\rho \sfe(\bu)$ is referred to as the internal energy.
The quantity $\calA\ni\sfp \mapsto \polR$ is the pressure defined by the equation of state (EOS) of interest where $\calA\subset\polR^{d+2}$ is some yet-to-be-defined admissible set in phase space (see \S\ref{sec:invariant_domain}). Throughout the paper, we will refer to $\sfp(\bu)$ as the ``oracle'' pressure.
\begin{remark}[Notation]
    It is common in the literature to be loose with notation regarding the equation of state.  More specifically,
    an equation of state can be defined by $p = p(\tau, s)$ where $s = s(\tau, e)$ is the specific entropy and $e = e(\tau, s)$ is  the specific internal energy.
    Similarly, one can define the equation of state by only using the conserved variable as an argument: $p = p(\bu)$. The reader is often left to deduce arguments based on context.
    While somewhat clunky, we choose to avoid this by defining unique symbols based on the arguments of the function. That is, $p(\tau, s)$ and $\sfp(\bu)$ will both represent the pressure and the arguments are clearly understood. This applies to the specific internal energy (\eg $e(\tau, s)$ and $\sfe(\bu)$) among other thermodynamic quantities.
\end{remark}
\subsection{The equation of state and thermodynamics} \label{sec:eos}
In this work, we assume that the underlying equation of state is complete; that is to say, the specific internal energy can be written as a function of specific volume and specific entropy: $e = e(\tau, s)$.
The other thermodynamic states such as pressure and temperature are obtained via: $p(\tau, s)\eqq -\partial_\tau e(\tau, s)$ and $T(\tau, s)\eqq \partial_s e(\tau, s)$, respectively, through the fundamental identity: $\diff e = -p \diff\tau + T \diff s$.
We define the isentropic bulk modulus by, $K \eqq -\tau \partial_\tau p(\tau,s)$, and the inverse of the specific heat capacity at constant volume, $c_v^{-1} \eqq \frac{1}{T} \partial_s T(\tau, s)$.
Note that if $\partial_\tau p(\tau, s) < 0$ and $\partial_s T(\tau, s) > 0$, then we can invert for $\tau$ and $s$ respectively; that is, $\tau = \tau(p, s)$ and $s = \cals(\tau, T)$.
Assuming this, we are now able to define the isothermal bulk modulus, $K_T \eqq -\tau \partial_\tau p(\tau, \cals(\tau, T))$, and the inverse of the specific heat capacity at constant pressure, $c_p^{-1} \eqq \frac{1}{T} \partial_s T(\tau(p,s), s)$.
These thermodynamic derivatives are valuable for analyzing properties and behavior of the equation of state and allows us to frame the joint convexity of $e(\tau, s)$ in the sense of~\citet[Eq.~2.33]{menikoff1989riemann}.
\begin{definition}[Convexity of $e(\tau,s)$]
    The specific internal energy, $e(\tau,s)$, is convex if the following is satisfied:
    \begin{equation}
        c_v^{-1} \geq c_p^{-1} > 0, \quad \text{ and } \quad K \geq K_T > 0.
    \end{equation}
\end{definition}
Note that we are using different notation for the bulk moduli than in \citep{menikoff1989riemann}.
For the purposes of the numerical method and all proofs contained in this paper, we only require positivity of the isentropic bulk modulus, $K = \tau \partial_\tau p(\tau, s) = \tau \partial^2_\tau e(\tau,s)$ since this yields real values for the sound speed: $c(\tau, s) = \sqrt{\tau K} = \sqrt{-\tau^2 \partial_\tau p(\tau,s)}$.
Hence we only require convexity of the \textit{isentropes}, $e_s(\tau)$.

Furthermore, we will require knowledge of the fundamental derivative:
\begin{equation} \label{eq:fundamental_derivative}
    G(\tau, s) \eqq -\frac12 \tau \frac{\partial^2_{\tau} p(\tau,s)}{\partial_\tau p(\tau,s)} = 1 - \frac{\tau}{c}\partial_\tau c(\tau,s),
\end{equation}
which will be used to ensure the invariant-domain preserving properties.
The fundamental derivative can be traced back to \citet{bethe} wherein analysis of $\partial^2_\tau p(\tau,s)$ was performed.
Later on, \citet{thompson1971fundamental}, introduced the unitless term, $G$, referring to it as the ``fundamental derivative''.
This quantity is essential for determining the wave structure in the Riemann problem as noted in~\citet{menikoff1989riemann}.
For more information regarding the fundamental derivative, see \citet{colonna2009computation}.
\begin{remark}[Smoothness properties] \label{rem:differentiability}
    For the purposes of Lemmas~\ref{lem:expansion_wave} and \ref{lem:shock_wave}, we assume that the equation of state is sufficiently smooth so that $K$ is continuous and differentiable a.e.~and that $G \in L^\infty_{\text{loc}}(\Real_+^2)$.
    This is the case for the Davis reactant EOS discussed in Section~\ref{sec:davis}.
    This assumption is generally true for analytic EOS; however, for tabular EOS, this assumption may fail.
    Specific analysis would need to be performed on the tabular EOS in order to determine its properties.
    See the tabular EOS remark in \citet[Sec.~4.7.2]{Menikoff2007};
    \citet{dilts2006consistent} also provides an insightful analysis regarding tabular EOS and a methodology to construct a thermodynamically consistent tabular approximation.
\end{remark}
\begin{assumption}[Thermodynamic assumptions] \label{def:thermo_consistency}
    We assume that the equation of state is complete and that the following assumptions are satisfied for all $s \geq \sfs_0$ where $\sfs_0 \eqq \min_{\bx \in D} (\sfs(\bu_0(\bx)))$:
    \begin{enumerate}
        \item[(i)] the energy isentrope, $e_s(\tau)$, is convex;
        \item[(ii)] the fundamental derivative satisfies: $\essinf G(\tau, s) > 1$.
    \end{enumerate}
\end{assumption}
Notice that condition~\textup{(i)} implies that the sound speed is real and hence the system~\eqref{euler} is always hyperbolic.
From \eqref{eq:fundamental_derivative}, condition~\textup{(ii)} implies that the isentropic speed of sound is always \textit{decreasing} along an isentrope.
While condition (ii) is somewhat restrictive, it is valid for low molecular complexity (LMC) fluids as described in \citet{guardone2024nonideal}.
\begin{definition}[3rd law of thermodynamics] \label{rem:third_law}
    The 3rd law of thermodynamics as stated by \cite{planck2013treatise}, tells us that the $T = 0$ isotherm corresponds to the $s=0$ isentrope.
    That is,
    \begin{equation}
        e_{\textup{cold}}(\tau) \eqq
        \lim_{s\to 0^+} e(\tau, s) = \lim_{T \to 0^+} e(\tau, \cals(\tau, T))
    \end{equation}
    This minimum curve is referred to as the ``cold'' curve.
\end{definition}

\begin{remark}[Positivity of temperature]
    We note that for an equation of state that satisfies the 3rd law of thermodynamics, positivity of temperature is equivalent to staying above the energy cold curve. We illustrate this fact in Figure~\ref{fig:macaw} with the Simple MACAW equation of state discussed in \S\ref{sec:simple_macaw}. We plot the cold curve (\ie the $T = 0$ isotherm) with a solid black line as well as two non-zero isotherms. We see that as we increase temperature, we move away from the cold curve.
\end{remark}

\begin{remark}[Ideal gas law]
    It should be noted that the 3rd law is not satisfied by every EOS.
    Most notably, the ideal gas law violates this law.
    Recall that the specific internal energy is given by, $e(\tau, s) = \tau^{-(\gamma-1)} \exp\big( \frac{s - s_0}{c_v} \big) = c_v T$.
    Taking the limit as $T \to 0^+$, the cold isotherm corresponds to the constant curve $e = 0$.
    However, for the 3rd law to be satisfied, this requires $s \to -\infty$.
\end{remark}

\begin{figure}
    \centering
    \begin{tikzpicture}
        \begin{axis}[
                width=10cm, height=7cm,
                xlabel={Specific volume, \(\tau\)},
                ylabel={Specific internal energy, \(e\)},
                title={Isotherms for the simple MACAW},
                legend pos=north west,
                legend cell align=left,
                grid=major,
            ]
            \addplot[domain=0.05:1,samples=200,thick,solid,black] {sie(x,0)};
            \addlegendentry{T=0}
            \addplot[domain=0.05:1,samples=200,thick,dashed,blue] {sie(x,2000)};
            \addlegendentry{T=2000}
            \addplot[domain=0.05:1,samples=200,thick,dotted,red] {sie(x,10000)};
            \addlegendentry{T=10000}
        \end{axis}
    \end{tikzpicture}
    \caption{Plots of the isotherms for the simple MACAW equation of state \citep{lozanoSimpleMACAW} using parameters given in Table~\ref{tab:macaw-parameters}.
        The $T=0$ isotherm is the cold curve since the simple MACAW EOS satisfies the 3rd law of thermodynamics.}\label{fig:macaw}
\end{figure}

\subsection{Admissible sets and invariant domains}\label{sec:invariant_domain}
We now supplement the system~\eqref{euler} with an admissible set in phase space. This is useful for constructing approximation techniques that yield physically relevant solutions.
The global admissible set for an equation of state which satisfies the third law of thermodynamics (Remark~\ref{rem:third_law}) is defined by,
\begin{equation} \label{eq:global_admissible_set}
    \calB \eqq \{ \bu \in \Real^{d+2} : \rho > 0, \, \rho\sfe(\bu) > \rho e_{\text{cold}}(\rho^{-1}) \},
\end{equation}
We further define an invariant region based on some initial minimum entropy, $s_0 > 0$:
\begin{equation}\label{eq:admissible}
    \calA_{s_0} \eqq \{ \bu \in \Real^{d+2} : \rho > 0, \, \rho\sfe(\bu) \geq \rho e_{s_{0}}(\rho^{-1}) \}.
\end{equation}
We see that $\calA_{s_0} \subset \calB$ since $T = \partial_s e(\tau, s) > 0$.
That is, $T > 0$ is equivalent to $e_{s_0}(\tau) > e_{\text{cold}}(\tau)$.
By Proposition~\ref{prop:concave_constraint}, the set~\eqref{eq:admissible} is composed of constraints that are at least quasiconcave, which it can then be characterized as an \textit{invariant region} in the sense of~\cite{chueh1977positively}.
Note, we do not characterize the admissible set as $T(\bu) > 0$, since $T(\bu)$ may not be a quasiconcave functional.
We refer to the set~\eqref{eq:admissible}, as either the admissible set, invariant region or invariant set.
We also write $\calA$ instead of $\calA_{s_0}$ for simplicity.
The goal of this work is to construct an approximation techniques that is invariant to~\eqref{eq:admissible}.

\begin{proposition}[Concave constraint] \label{prop:concave_constraint}
    Under Assumption~\ref{def:thermo_consistency} (i), the functional $\rho \sfe(\bu) - \rho e_{s_0}(\tau(\bu))$ is a concave function of $\bu$.
\end{proposition}

\begin{proof}
    It is well know that $\rho \sfe(\bu)$ is a concave function of $\bu$.
    Therefore, we only need show that $\rho e_{s_0}(\tau(\bu))$ is convex.
    Differentiating twice with respect to $\rho$, we have,
    \begin{equation*}
        \partial^2_\rho (\rho e_{s_0}(\tau)) = \frac{1}{\rho^3} \partial^2_\tau e_{s_0}(\tau) > 0.
    \end{equation*}
    Therefore, $\rho\sfe(\bu) - \rho e_{s_0}(\tau(\bu))$ is concave.
\end{proof}

\begin{remark}[Specific internal energy]
    Note that the condition $\sfe(\bu) > e_{s_0}(\tau)$ does not necessarily imply that the specific internal energy must be positive or even that $\sfe(\bu) > 0$ is satisfactory.
    The physical constraint is positivity of the temperature, $T > 0$.
    For certain equations of state, the reference isentrope $e_{s_0}(\tau)$ can be negative on a particular interval for $\tau$.
    A simple example is when the equation of state is the stiffened gas law: $e(\tau,s) = q + p_\infty \tau + C_0(s) \tau^{-(\gamma-1)}$ where $C_0(s)$ is a \textit{positive} constant depending on the exponential of $s$.
    Here, $q\in\Real$ is a reference specific internal energy and $p_\infty \geq 0$ a reference pressure.
    The temperature is given by $T = \frac{e - q - p_\infty \tau}{c_v}$, hence the positivity of temperature is guaranteed when $e > q + p_\infty\tau$.
    Therefore, for $q < 0$, the specific internal energy can be negative while still maintaining $T > 0$.
    Plots of the zero isotherm and several isentropes are provided in Figure~\ref{fig:stiff_gas_isentropes}.
\end{remark}
\begin{figure} \label{fig:stiff_gas_isentropes}
    \centering
    \begin{tikzpicture}
        \begin{axis}[
                width=10cm, height=7cm,
                xlabel={Specific volume, \(\tau\)},
                ylabel={Specific internal energy, \(e\)},
                title={Isentropes for the stiffened gas EOS},
                legend pos=north west,
                legend cell align=left,
                grid=major,
            ]
            \addplot[domain=0.005:1,samples=200,thick,solid,black] {stiff_isen(x,0)};
            \addlegendentry{$C_0=0$}
            \addplot[domain=0.005:1,samples=200,thick,dashed,blue] {stiff_isen(x,30)};
            \addlegendentry{$C_0=30$}
            \addplot[domain=0.005:1,samples=200,thick,dotted,red] {stiff_isen(x,100)};
            \addlegendentry{$C_0=100$}
        \end{axis}
    \end{tikzpicture}
    \caption{Plots of the isentropes for the stiffened gas equation of state with $q = -700$, $p_\infty = 1000$.
        The $T=0$ isotherm occurs at $C_0 = 0$ which occurs when $s \to -\infty$ (3rd law violating).}
\end{figure}
\begin{remark}[Entropy solutions]
    Regarding entropy solutions to the model~\eqref{euler}, it was proven in~\cite{tadmor1986minimum} that
    \begin{equation}
        \sfs(\bu(\bx,t)) \geq \essinf_{\Vert \by \Vert_{\ell^2} \leq R + t v_{\max}} \sfs(\bu_0(\by)), \quad \text{ for } \Vert\bx\Vert_{\ell^2} \leq R,
    \end{equation}
    for a polytropic gas equation of state where $v_{\max} \eqq \sup_{(\bx, s) \in D \times [0,t)} \Vert \bv(\bx, s) \Vert_{\ell^2}$.
    While a result of this nature has yet to be proven for arbitrary convex equations of state, we suspect it to be true provided the equation of state satisfies Assumption~\ref{def:thermo_consistency}.
\end{remark}

%
\section{The Riemann problem}\label{sec:riemann_problem}
The numerical approximation technique considered in this work relies on local wave speed estimates in the Riemann problem. We recall its definition here.
Let $\calA\subset\polR^m$ be a convex admissible set for the solution to the Riemann problem.
Let $\bn\in\polS^{d-1}(\bm{0},1)$ be a unit vector.
Given a pair of states $(\bw_L, \bw_R)\in\calA\times\calA$, the one-dimensional Riemann problem in the direction of $\bn$ is defined by the initial value problem:
\begin{equation}\label{eq:riemann_problem}
    \partial_t\bw + \partial_x(\polg(\bw)\bn) = \bm{0},\quad
    \bw(x, 0) =
    \begin{cases}
        \bw_L & \text{if }x < 0, \\
        \bw_R & \text{if }x > 0,
    \end{cases}
\end{equation}
where $\polg : \Real^m \to \Real^d$ is the flux.
Assume that $\bw(x, t; \bw_L, \bw_R, \bn)$ is a weak solution to \eqref{eq:riemann_problem} and that $\bw(x, t; \bw_L, \bw_R, \bn) \in \calA$ for $\text{a.e. } x \in \Real$. For the sake of simplicity, we will interchange the notation $\bw(x, t; \bw_L, \bw_R, \bn)$ with $\bw(x, t)$.
It is well-known that constructing numerical methods based on Riemann averages is essential for ensuring robustness of numerical solutions (see:~\citep{harten1983upstream, guermond2016, nessyahu1990non}). We define the following auxiliary state:
\begin{equation}\label{eq:auxiliary_state}
    \overline{\bw}(\lambda; \bw_L,\bw_R, \bn) \eqq \frac12(\bw_R + \bw_L) - \frac{1}{2\lambda}(\polg(\bw_R)\bn - \polg(\bw_L) \bn).
\end{equation}
We then have the following classical result:
\begin{lemma}[Riemann average] \label{lem:classic_bar_state_result}
    Let $\lambda_L$ and $\lambda_R$ denote the smallest and largest wave speeds bounding the Riemann fan, respectively.
    If $\hlambda \geq \max(|\lambda_L|, |\lambda_R|)$ and $t\hlambda \leq \frac12$ then
    \begin{equation}\label{eq:riemann_average}
        \overline{\bw}(\hlambda; \bw_L,\bw_R, \bn) = \int_{-\frac12}^{\frac12}\bw(x, t; \bw_L,\bw_R, \bn) \diff x,
    \end{equation}
    is average solution to the Riemann problem
\end{lemma}

\begin{proof}
    The proof is provided in~\citep[Lem.~2.1]{guermond2019}.
\end{proof}

The usefulness of Lemma~\ref{lem:classic_bar_state_result} offers an algebraic means of computing the average of the solution to the Riemann problem (assuming the maximum wave speed is known).
If $\bw \in \calA$ pointwise a.e.~then it immediately follows that $\overline{\bw}(\hlambda) \in \calA$.
For a reference and collection of useful results regarding the auxiliary state $\overline{\bw}$, see:~\citep[Lem.~3.2]{clayton_2022}.

\section{Approximation technique}\label{sec:approx_details}

In this section, we describe the algebraic structure for a generic approximation technique that can be made invariant-domain preserving if the equation of state satisfies Assumption~\ref{def:thermo_consistency}.
Let $t^n$, $n\in\polN$, denote the current discrete time and let $\dt$ denote the time step size. Let $\{\bsfU_i\upn\}_{i\in\calV}$ denote a collection of states used to describe the approximation of the conserved variable $\bu$ at $t^n$. Here, $\calV$ denotes the index set used to enumerate the spatial degrees of freedom. We further assume that for each $i\in\calV$ that $\bsfU_i\upn\in\calA\subset\polR^{d+2}$.

We also define an index set $\calI(i)$ used to enumerate the degrees of freedom that locally interact with $i\in\calV$.
We then introduce the discrete auxiliary states $\overline{\bsfU}^n_{ij}(\lambda_{ij})$ (which are the discrete analog to~\eqref{eq:auxiliary_state}) at every $t^n$ from the underlying spatial discretization.
The goal of the paper is to find $\hlambda_{ij}$ so that $\overline{\bsfU}^n_{ij}(\hlambda_{ij}) \in \calA$ for all $i \in \calV$ and $j \in \calI(i)$.
One does not necessarily require $\hlambda_{ij} \geq \max(\lambda_i, \lambda_j)$ and finding the true max wave speed is not realistically feasible for an arbitrary EOS; we only desire a wave speed which guarantees $\overline{\bsfU}^n_{ij}(\hlambda_{ij}) \in \calA$.

The general numerical approximation we consider in this work is one that can be written in the following convex combination:
\begin{equation} \label{eq:numerical_method}
    \bsfU_i\upnp = \sum_{k\in\calI(i)}\omega_{ik}\upn\bsfU_k\upn + \sum_{j\in\calI(i)\setminus\{i\}}\mu_{ij}\upn \overline{\bsfU}^n_{ij}(\hlambda_{ij}).
\end{equation}
where for each $i\in\calV$ the weights $\omega_{ik}\upn$ and $\mu_{ij}\upn$ satisfy $(\sum_{k\in\calI(i)}\omega_{ik}\upn + \sum_{j\in\calI(i)\setminus\{i\}}\mu_{ij}\upn) = 1$ and $\omega\upn_{ik}, \mu\upn_{ij} \in [0,1]$ for all $k\in\calI(i)$ and $j\in\calI(i)\setminus\{i\}$.
These weights encode the underlying spatial discretization.
\begin{proposition}[Invariant-domain preserving]
    Assume that $\bsfU_i\upn\in\calA$ at time $t\upn$ for all $i\in\calV$.
    Further assume that there exists $\hlambda_{ij} > 0$ such that $\overline{\bsfU}^n_{ij}(\hlambda_{ij})\in\calA$ for each $i\in\calV$ and all $j\in\calI(i)$.
    Then, $\bsfU_i\upnp\in\calA$. That is, the scheme~\eqref{eq:numerical_method} is invariant-domain preserving.
\end{proposition}
\begin{proof}
    Since $\bsfU_i\upn, \overline{\bsfU}^n_{ij}(\hlambda_{ij}) \in\calA$ and $\calA$ is convex, the convex combination defined by~\eqref{eq:numerical_method} also belongs to $\calA$.
\end{proof}

The goal of this work is to provide a methodology for computing wave speed estimates, $\{\hlambda_{ij}\}$, that enforces $\overline{\bsfU}^n_{ij}(\hlambda_{ij}) \in \calA$ where $\calA$ is defined in~\eqref{eq:admissible}.
Since it is highly non-trivial, or maybe even impossible, to find the exact maximum wave speed for every equation of state that satisfies Assumption~\eqref{def:thermo_consistency}, we adopt the \textit{extended Riemann problem} methodology introduced in~\cite{clayton_2022} to instead compute $\hlambda_{ij}$ that is computationally more feasible and still provides the desired properties (\ie $\overline{\bsfU}^n_{ij}(\hlambda_{ij}) \in \calA$).

\begin{remark}[Discrete entropy inequality]
    It is not necessarily the case that an invariant-domain preserving method satisfies a discrete entropy inequality.
    That it to say, an IDP method could still converge to the wrong solution.
    This is discussed in \citet[Sec.~5.1]{guermond2016}.
    Finding generic weights $\mu_{ij}\upn$ and $\omega_{ik}\upn$ that satisfy such a discrete entropy inequality is out of the scope of this paper.
\end{remark}
\section{The Extended Riemann problem} \label{sec:extended_riemann}
The extended Riemann problem can be thought of as a simpler, surrogate model that encodes information of the original Riemann problem~\eqref{eq:riemann_problem} through \textit{interpolation} of thermodynamic quantities via a local stiffened gas equation of state (see: Lemma~\ref{lem:interpolation}).
Furthermore, the auxiliary state of the extended Riemann problem \textit{contains} the auxiliary state of the oracle Riemann problem (see: equation~\eqref{eq:extended_bar_states}).
This fact will allow us to analytically derive a maximum wave speed $\tlambda_{ij}^{\max}$ for the extended Riemann problem which will enforce the desired invariant domain properties on the original auxiliary state.
We then prove in Lemma~\ref{lem:invariant_domain} that the auxiliary state of the oracle Riemann problem using $\tlambda^{\max}_{ij}$ (or an upper bound on $\tlambda^{\max}_{ij}$) satisfies the minimum principle on the oracle physical entropy.
\subsection{Set up}
Let $Z\in\{i,j\}$ and $\bsfU_Z\upn\eqq(\rho_Z, \bsfM_Z, \sfE_Z)^{\mathsf{T}}$ be the discrete data at $t^n$.
Let $\bn_{ij}$ denote some unit normal and can be thought of as the discrete counterpart to $\bn$ defined in \S\ref{sec:riemann_problem}.
We set $e_Z\eqq \sfe(\bsfU_Z\upn)$ to be the specific internal energy data. Let $p_Z\eqq\sfp(\bsfU_Z\upn)$ be the oracle pressure data and $K_Z\eqq \sfK(\bsfU_Z\upn)$ the oracle isentropic bulk modulus data. Let $s_{\min} \eqq \min_{Z\in\{L,R\}} \sfs(\bsfU^n_Z)$ be the minimum entropy oracle data.
We further define a local total energy quantity $\calE_Z\eqq \sfE_Z - \frac{\|\bsfM_Z - (\bsfM_Z\SCAL\bn_{ij})\bn_{ij}\|^2_{\ell^2}}{2 \rho_Z}$.
Note also that $\calE_Z = \rho_Z e_Z +  \frac{\|(\bsfM_Z\SCAL\bn_{ij})\bn_{ij}\|^2_{\ell^2}}{2 \rho_Z}$. Although the ``left'' and ``right'' states of interest are the discrete states $(\bsfU_i\upn, \bsfU_j\upn)$, respectively, we introduce the following \textit{continuous} extended Riemann problem for better readability.

Let $\tbu\eqq(\rho, m, \calE, \calP_\infty, \calQ)^{\mathsf{T}}$. Then we consider the following one-dimensional extended Riemann problem:
\begin{subequations} \label{eq:full_extended_RP}
    \begin{align}\label{eq:extended_RP}
         & \partial_t\begin{pmatrix}
                         \rho         \\
                         m            \\
                         \calE        \\
                         \calP_\infty \\
                         \calQ
                     \end{pmatrix} + \partial_x \begin{pmatrix}
                                                    m                                      \\
                                                    \frac{m^2}{\rho} + \pinterp(\tbu)      \\
                                                    \frac{m}{\rho}(\calE + \pinterp(\tbu)) \\
                                                    \frac{m}{\rho} \calP_\infty            \\
                                                    \frac{m}{\rho} \calQ
                                                \end{pmatrix} = 0,  \quad \text{ where }                                                 \\
        %
         & \pinterp(\tbu) = (\gamma - 1) (\calE - \frac{m^2}{2 \rho} - \calQ) - \frac{\gamma \calP_\infty}{\rho},\label{eq:extended_pressure}
    \end{align}
\end{subequations}
with local data $\tbu_Z \eqq (\rho_Z, \bsfM_Z\SCAL\bn_{ij}, \calE_Z, \calP_{\infty, Z}, \calQ_{Z})^{\mathsf{T}}$.
Here, we call $\pinterp(\tbu)$ the interpolant pressure.
The extended quantities $\calP_\infty(x, t)$ and $\calQ(x, t)$ are surrogate scalars transported by the velocity $\frac{m}{\rho}$. We define their respective primitive counterparts by: $\calp_\infty\eqq\frac{\calP_\infty}{\rho}$ and $\calq\eqq\frac{\calQ}{\rho}$.
Then using the primitive variables in the interpolant pressure definition, we see that $\pinterp(\tbu) = (\gamma - 1)(\rho e - \rho \calq) - \gamma \calp_\infty$ which resembles the stiffened gas pressure law.
However, note that here, $\calp_\infty$ and $\calq$ are not the usual constants and instead depend on $(x, t)$.
We set the left/right data of the extended primitive quantities by:
\begin{subequations}\label{eq:local_data}
    \begin{align}
        \gamma           & = 1+\epsilon, \label{eq:gammaz}                                                            \\
        \calp_{\infty,Z} & = \frac{K_{Z}}{\gamma} - p_Z \label{eq:pinfz},                                             \\
        \calq_Z          & = e_Z - \tau_Z \Big( \frac{p_Z + \gamma \calp_{\infty,Z}}{\gamma - 1} \Big), \label{eq:qz}
    \end{align}
\end{subequations}
for some $\epsilon > 0$.
(We will see in the proof of Lemma~\ref{lem:expansion_wave} that $\epsilon$ must satisfy $\epsilon < 2(\essinf_{\tau \geq \tau_Z} G(\tau,s_Z) - 1)$; however, the choice of $\epsilon$ becomes irrelevant by Section~\ref{sec:robust_wave_speed_estimate}.)
The definition of the local data defined above is chosen so that the extended Riemann problem interpolates relevant thermodynamic quantities of the oracle.
We summarize below.
\begin{lemma}[Thermodynamic interpolation] \label{lem:interpolation}
    Let $p_Z$ be the pressure oracle data and $K_Z$ be the isentropic bulk modulus data, for $Z\in\{i,j\}$.
    Let $\pinterp(\tbu)$ be the interpolation pressure defined by \eqref{eq:extended_pressure}, and $\gamma$, $\calp_{\infty,Z}$, and $\calq_Z$ defined by \eqref{eq:local_data}, respectively.
    Let $\widetilde{\calK}(\tbu)\eqq \gamma(\pinterp(\tbu) + \calp_\infty)$ be the interpolatory isentropic bulk modulus.
    Then
    \begin{equation}
        \pinterp(\tbu_Z) = p_Z, \quad \text{ and } \quad
        \calK_Z(\tbu_Z) = K_{Z}.
    \end{equation}
\end{lemma}
\begin{proof}
    The result is shown through direct computation.
    Note: $\calp(\tbu_Z) = (\gamma - 1) (\rho_Z e_Z - \rho_Z \calq_Z) - \gamma \calp_{\infty, Z}$.
    Substituting~\eqref{eq:qz}, we see that $\calp(\tbu_Z) = p_Z$.
    Similarly, interpolation of the isentropic bulk modulus follows from \eqref{eq:pinfz}.
\end{proof}
Note that interpolation of the (isentropic) bulk modulus is the same as the interpolation of the (isentropic) sound speed since $K_Z = \rho_Z c_Z^2$.
\begin{lemma}[Energy interpolation]
    The specific internal energy for the interpolatory equation, $\widetilde{\cale}(\calp, \rho, \calq, \gamma,  \calp_{\infty}) = \frac{\calp + \gamma p_\infty}{(\gamma - 1) \rho} + \calq$, satisfies,
    \begin{equation}
        \widetilde{\cale}(\calp_Z, \rho_Z, \calq_Z, \gamma, \calp_{\infty,Z}) = e_Z.
    \end{equation}
\end{lemma}
\begin{proof}
    This is an immediate consequence from the interpolation of the pressure, $\pinterp(\tbu_Z) = p_Z$.
\end{proof}
The fact that the specific internal energy and the pressure are interpolated is essential for the proof of the invariant domain preserving properties in Section~\ref{sec:invariant_domain}.

We note that the idea of interpolating a general equation of state with a stiffened gas pressure law can be traced back to \citep[Sec.~VII]{menikoff1989riemann}. A more recent approach that also uses the stiffened gas law for interpolating the equation of state can be seen in~\cite[Sec.~2.1]{wang2023stiffened}.
\subsection{Solution to the extended Riemann problem}
Solving the extended Riemann problem follows the standard methodology as proposed by~\citet{Lax_1957} and was detailed in~\citep{clayton_2022} for an ideal gas interpolatory pressure law with an additional equation for a variable $\gamma$ only.
This idea was originally inspired by \citet{abgrall2001computations}.
Since the technique here is similar to~\citep{clayton_2022}, we skip the full details and instead summarize the results.
It can be shown that the solution to the Riemann problem is composed of two genuinely non-linear waves (either expansions or shocks) and three contact waves for $\rho$, $\calP_\infty$ and $\calQ$.
A consequence of the extended quantities $\calP_\infty$ and $\calQ$ being contact waves is that the interpolant pressure reduces to a simple stiffened gas law on the left of the contact: $\pinterp(\tbu) = (\gamma - 1) (\rho e - \rho \calq_L) - \gamma \calp_{\infty, L}$ and $\pinterp(\tbu) = (\gamma - 1) (\rho e - \rho \calq_R) - \gamma \calp_{\infty, R}$ to the right of the contact.
(Here, we have interchanged $i$ with $L$ and $j$ with $R$ to emphasize \textit{left} and \textit{right}, respectively).
This fact makes it feasible to compute a maximum wave speed in the extended Riemann problem since the extended variables do not influence the left/right genuinely non-linear waves.
We now motivate why working with the extended Riemann problem is sufficient for enforcing $\overline{\bsfU}_{LR}^n(\tlambda_{LR}) \in \calA$.

Since the oracle pressure is interpolated in the extended Riemann problem (as shown in Lemma~\ref{lem:interpolation}), this implies that the local left/right flux of the extended Riemann problem is given by:
\[\widetilde{\polf}(\tbu_Z)\bn = (\polf(\bu_Z)\bn, v_Z \calP_{\infty,Z}, v_Z \calQ_Z)^{\mathsf{T}}.\]
That is to say, the extended flux evaluated at $\tbu_Z$ contains the flux of the oracle Riemann problem evaluated at $\bu_Z$.
Let $\overline{\bu}_{LR}$ denote the auxiliary quantity \eqref{eq:auxiliary_state}.
\begin{equation} \label{eq:extended_bar_states}
    \overline{\tbu}_{LR}(\htlambda_{LR}; \tbu_L, \tbu_R, \bn)
    \eqq \begin{pmatrix}
        \overline{\bu}_{LR}(\htlambda_{LR}; \bu_L, \bu_R, \bn)                                                                   \\
        \frac12(\calP_{\infty,L} + \calP_{\infty,R}) - \frac{1}{2\htlambda_{LR}} (v_R \calP_{\infty, R} - v_L \calP_{\infty, L}) \\
        \frac12(\calQ_L + \calQ_R) - \frac{1}{2\htlambda_{LR}} (v_R \calQ_R - v_L \calQ_L),
    \end{pmatrix},
\end{equation}
where $\htlambda_{LR}$ is an upper bound to the maximum wave speed to the extended Riemann problem.
This shows that the average of the solution to the extended Riemann problem \textit{contains} the auxiliary quantity $\overline{\bu}_{LR}$.
Therefore, the auxiliary state in \eqref{eq:extended_bar_states} satisfies the identity \eqref{eq:riemann_average} (for $t\htlambda_{LR} \leq \frac12$); that is,
\begin{equation} \label{eq:extended_bar_state_average}
    \overline{\tbu}_{LR}(\htlambda_{LR}; \tbu_L, \tbu_R, \bn)
    = \int_{\frac12}^{\frac12} \tbu(x, t; \tbu_L, \tbu_R, \bn) \diff x,
\end{equation}
where $\tbu$ is the solution to the \textit{extended} Riemann problem.
The important result of this methodology is that
\begin{equation}
    \widetilde{\sfe}(\overline{\tbu}_{LR}(\htlambda_{LR})) = \sfe(\overline{\bu}_{LR}(\htlambda_{LR})),
\end{equation}
where $\widetilde{\sfe}(\tbu) = \sfe(\bu)$ is simply the specific internal energy of the extended Riemann problem and it remains unchanged.
Thus, one is able to analyze the much simpler extended Riemann problem to prove results regarding $e$.
This concept was first introduced in \citep{clayton_2022}.
We wish to further emphasize that $\overline{\bu}_{LR}(\htlambda_{LR})$ is the auxiliary state based on the maximum wave speed of the extended Riemann problem and is not necessarily the maximum wave speed to the oracle Riemann problem.
We would like to also emphasize that the discrete auxiliary state is simply $\overline{\bsfU}^n_{LR}(\htlambda_{LR}) = \overline{\bu}_{LR}(\htlambda_{LR})$ derived from the continuous solution.
We now give the maximum wave speed in the extended Riemann problem, but a more robust wave speed estimation is described in Section~\ref{sec:robust_wave_speed_estimate}.

\subsection{The maximum wave speed}\label{sec:max_wave_speed}
Since the interpolatory pressure is just a stiffened gas pressure law left and right of the contact, it can be shown the leftmost and rightmost wave speeds in the extended Riemann problem are defined by:
\begin{subequations}
    \begin{align}
        \lambda_L(\calp^*) & \eqq v_L - \mathcal{c}_L\sqrt{\frac{\gamma + 1}{2\gamma} \max\Big( \frac{\calp^* - p_L}{\calp_{\infty,L} + p_L}, 0 \Big) + 1}, \label{eq:left_wave_speed}   \\
        \lambda_R(\calp^*) & \eqq v_R + \mathcal{c}_R\sqrt{\frac{\gamma + 1}{2\gamma} \max \Big( \frac{\calp^* - p_R}{\calp_{\infty,R} + p_R}, 0 \Big) + 1}, \label{eq:right_wave_speed}
    \end{align}
\end{subequations}
where $\mathcal{c}_Z \eqq \sqrt{\gamma\frac{p_Z + \calp_{\infty,Z}}{\rho_Z}} = \sqrt{\tau_Z K_Z}$ for $Z \in \{L, R\}$ (see:~\cite[Sec.~4.2.1]{clayton2023robust}). Here, $\gamma\eqq 1 + \epsilon$.
The value $\calp^*$ is the root to the following nonlinear equation,
\begin{equation} \label{eq:varphi}
    \varphi(\calp) \eqq f_L(\calp) + f_R(\calp) + v_R - v_L = 0,
\end{equation}
where
\begin{equation}
    f_Z(\calp) = \begin{cases}
        \frac{2\mathcal{c}_Z}{\gamma-1} \Big( \Big( \frac{\calp + \calp_{\infty,Z}}{p_Z + \calp_{\infty,Z}} \Big)^{\frac{\gamma-1}{2\gamma}} - 1 \Big), \quad \text{ if } \calp < p_Z, \\
        (\calp - p_Z) \sqrt{\frac{A_Z}{\calp + \calp_{\infty,Z} + B_Z}}, \quad \text{ if } \calp \geq p_Z,
    \end{cases}
\end{equation}
with $A_Z = \frac{2\tau_Z}{\gamma+1}$ and $B_Z = \frac{\gamma-1}{\gamma+1}(p_Z + p_{\infty,z})$.
Solving for $\calp^*$ requires an iterative method, and an efficient process for solving is seen in \citet{Guermond_fast}.

In the next section we show that the use of the wave speeds \eqref{eq:left_wave_speed} and \eqref{eq:right_wave_speed}, guarantees us that the average of the Riemann solution preserves the minimum principle on the entropy.
The proof of Lemma~\ref{lem:expansion_wave} will require $\epsilon$ to satisfy $\epsilon < 2(\essinf_{\tau \geq \tau_Z} G(\tau,s_Z) - 1)$.
In Section~\ref{sec:robust_wave_speed_estimate}, we circumvent this condition and show that one is able to take the limit as $\epsilon \to 0^+$ which simplifies the wave speed computation.
\subsection{Invariant domain properties}
It was shown in~\citep{clayton_2022} that the average of the extended Riemann problem satisfies the positive density condition. We extend this result and now show that the average of the extended Riemann problem solution will also satisfy the minimum principle for the oracle entropy (provided the oracle equation of state satisfies Assumption~\eqref{def:thermo_consistency}).
This is the main result of the paper.

The proof outline consists of showing that the specific internal energy on each wave of the extended Riemann problem remains on or above the corresponding isentrope.
For this, we deduce that the average of the solution to the extended Riemann problem remains above the minimum isentrope and hence preserves the minimum principle on the specific entropy.

\begin{lemma}[Expansion wave] \label{lem:expansion_wave}
    Let $Z \in \{L, R\}$ and $\gamma = 1 + \epsilon$.
    Assume the $Z$-wave in the extended Riemann problem is an expansion wave.
    If
    \begin{equation} \label{eq:essential_G}
        \essinf_{s \geq s_Z, \tau \geq \tau_Z} G(\tau,s) > 1,
    \end{equation}
    then there exists an $\epsilon_0 > 0$ such that for all $\epsilon \in (0,\epsilon_0)$, the isentrope of the interpolant stiffened gas EOS is above the isentrope of the oracle which intersects at $(\tau_Z, e_Z)$.
    That is,
    \begin{equation}
        e^{\textup{stiff}}_s(\tau(\xi)) \geq e^{\textup{oracle}}_s(\tau(\xi))
    \end{equation}
    where $\tau(\xi)$ (with $\xi \eqq x/t$) is the self-similar solution of the specific volume on the expansion wave.
\end{lemma}
\begin{proof}
    Throughout this proof, we use the subscript, ${}_s$, to denote a thermodynamic quantity on its respective isentrope.
    We start by first noting that $e^{\text{stiff}}_s(\tau_Z) = e^{\text{oracle}}_s(\tau_Z)$; this follows from \eqref{eq:qz}.
    However, this does not imply that the two entropies are the same.
    On an expansion wave, the specific volume is increasing; that is, $\tau(\xi) \in [\tau_Z, \infty)$.
    Since we don't know the exact value that $\tau(\xi)$ approaches to on the expansion wave, the result will be proven for the full range $\tau(\xi) \in [\tau_L, \infty)$.
    To simplify notation, we drop the dependency on $\xi$.

    We would like to emphasize again that specific entropy increases with respect to specific internal energy.
    Hence if $e^{\text{stiff}}_s(\tau) \geq e^{\text{oracle}}_s(\tau)$, then the energy belongs to an isentrope with higher specific entropy.  Figure~\ref{proof:isentrop_comparison} provides a graphical interpretation of this idea.

    \begin{figure}
        \centering
        \includegraphics[width=0.5\linewidth]{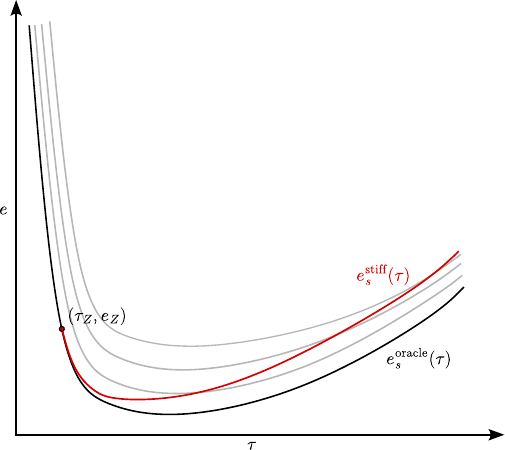}
        \caption{A visual depiction of the proof objective. That is, the stiffened gas isentrope remains on or above the oracle isentrope.
            Curves in gray depict isentropes of the oracle at higher specific entropies.}
        \label{proof:isentrop_comparison}
    \end{figure}

    To prove the result, we set $f(\tau) \eqq e^{\text{stiff}}_s(\tau) - e^{\text{oracle}}_s(\tau)$ and note that $f(\tau_Z) = 0$.
    From Lemma~\ref{lem:interpolation}, we have that $f'(\tau_Z) = f''(\tau_Z) = 0$.
    Therefore, we need only show that $f''(\tau) \geq 0$ for all $\tau \geq \tau_Z$.
    Note that $f''(\tau) \geq 0$ is equivalent to the following
    \begin{equation} \label{eq:bmod_inequal}
        K_s^{\text{stiff}}(\tau) \geq K_s^{\text{oracle}}(\tau),
    \end{equation}
    where $K_s^{\text{stiff}}$ and $K_s^{\text{oracle}}$ are the isentropic bulk moduli on their respective isentropes.
    Along an isentropic expansion, the pressure for the stiffened gas law is $\pstiff_s(\tau) = \frac{C_Z}{\tau^{\gamma}} - \calp_{\infty,Z}$ where $C_Z$ is some constant dependent on the specific entropy (see~\citep[eqn.~(4.44)]{clayton2023robust}).
    Recalling that $K^{\text{stiff}}_s(\tau) = \gamma(\pstiff_s(\tau) + \calp_{\infty,Z})$, the inequality~\eqref{eq:bmod_inequal} is equivalent to
    \[
        1 \geq \frac{K_s^{\text{oracle}}(\tau) \tau^{\gamma}}{\gamma C_Z}.
    \]
    Since equality holds for $\tau = \tau_L$, we only need to show that $K_s^{\text{oracle}}(\tau) \tau^{\gamma}$ is a decreasing function.
    Note that $K_s^{\text{oracle}}(\tau) \in W^{1,\infty}_{\text{loc}}(\Real_+)$ and therefore, we shall show that $\partial_\tau(K_s^{\text{oracle}}(\tau) \tau^{\gamma} ) \leq 0$ almost everywhere.

    We have that $\frac{\diff}{\diff \tau} (K_s^{\text{oracle}}(\tau) \tau^{\gamma}) = \tau^{\gamma} \partial_\tau K_s^{\text{oracle}}(\tau) + \gamma \tau^{\gamma - 1} K_s^{\text{oracle}}(\tau)$.
    Using the identity, $\partial^2_{\tau\tau} p(\tau, s) = 2KG/\tau^2$, we have
    \[
        \frac{\diff}{\diff \tau}(K_s^{\text{oracle}}(\tau) \tau^{\gamma}) = \tau^{\gamma - 1} K_s^{\text{oracle}}(\tau) \Big( (\gamma + 1)  - 2G_s^{\text{oracle}}(\tau) \Big),
    \]
    for almost every $\tau \in (\tau_Z, \infty)$.
    Substituting in $\gamma = 1+\epsilon$, we have the following inequality,
    \begin{equation*}
        \begin{split}
            \frac{\diff}{\diff \tau} \big(K_s^{\text{oracle}}(\tau) \tau^{1+\epsilon}\big) & = \tau^{\epsilon} K_s^{\text{oracle}}(\tau) \Big( 2+\epsilon - 2G_s^{\text{oracle}}(\tau) \Big)                                                       \\
                                                                                           & \leq 2\tau^{\epsilon} K_s^{\text{oracle}}(\tau) \Big( \frac{\epsilon}{2} - \big(\essinf_{\tau \geq \tau_Z} G_s^{\text{oracle}}(\tau) - 1 \big) \Big).
        \end{split}
    \end{equation*}
    Thus, from assumption \eqref{eq:essential_G}, we can select $\epsilon_0 \eqq 2(\essinf_{\tau\geq\tau_Z} G_s^{\text{oracle}}(\tau) - 1)$ and hence for all $\epsilon \in (0,\epsilon_0)$, we have that $\frac{\diff}{\diff \tau} (K_s^{\text{oracle}}(\tau) \tau^{1+\epsilon}) \leq 0$.
    Therefore, $e^{\text{stiff}}_s(\tau) \geq e^{\text{oracle}}_s(\tau)$ for $\tau \in [\tau_Z, \infty)$.
\end{proof}
\begin{lemma}[Shock wave] \label{lem:shock_wave}
    Let $Z \in \{L, R\}$.
    Assume the $Z$-wave in the extended Riemann problem is a shock wave.
    If
    \begin{equation} \label{eq:shock_G}
        \essinf_{s \geq s_Z, \tau \leq \tau_Z} G(\tau,s) > -\frac32,
    \end{equation}
    then the specific internal energy of the stiffened gas EOS across the shock wave is above the isentrope of the oracle which intersects at $(\tau_Z, e_Z)$.
    That is,
    \begin{equation}
        e^{\textup{stiff}}_\calH(\tau^*_Z) \geq e^{\textup{oracle}}_s(\tau^*_Z)
    \end{equation}
    where $\tau^*_Z < \tau_Z$ is the specific volume across the shock wave and $e^{\textup{stiff}}_{\calH}(\tau)$ is the specific internal energy on the Hugoniot curve, $\calH(\tau, p) = 0$.
\end{lemma}
\begin{proof}
    Recall that the Hugoniot curve can be written as
    \begin{equation}
        \calH(\tau, p) \eqq e(\tau, p) - e_Z + \frac12 (p + p_Z)(\tau - \tau_Z) = 0,
    \end{equation}
    see \citet[Chpt.~III. Eq.~2.13]{godlewski_raviart_2021}.
    This equation is agnostic of the equation of state.
    The pressure across the shock wave for the stiffened gas law is:
    \begin{equation}
        p_{\calH}^{\text{stiff}}(\tau)
        \eqq (p_Z + \calp_{\infty,Z}) \frac{\frac{\gamma - 1}{\gamma + 1} - \frac{\tau_Z}{\tau}}{\frac{\gamma - 1}{\gamma + 1} \frac{\tau_Z}{\tau} - 1} -
        \calp_{\infty,Z},
    \end{equation}
    for $\tau \in (\frac{\gamma-1}{\gamma+1}\tau_Z, \tau_Z)$, (see~\cite[Sec.~4.2.1]{Toro_book} for the general process as well as~\citep[eqn.~4.32]{clayton2023robust}).
    From this, we see that,
    \begin{equation}
        e^{\text{stiff}}_{\calH}(\tau) = e_Z + \frac12 (p^{\text{stiff}}_{\calH}(\tau) + p_Z)(\tau - \tau_Z).
    \end{equation}
    Just as in the proof of Lemma~\ref{lem:expansion_wave}, we have that $e^{\text{stiff}}_{\calH}(\tau_Z) = e_Z$ and
    $-\frac{\diff}{\diff \tau} e^{\text{stiff}}_{\calH}(\tau_Z) = p_Z = -\partial_{\tau} e_s^{\text{oracle}}(\tau_Z)$.
    Therefore, it suffices to show that $\frac{\diff^2}{\diff \tau^2} e^{\text{stiff}}_{\calH}(\tau) \geq \partial^2_{\tau} e^{\text{oracle}}_s(\tau)$ for all $\tau \in (\frac{\gamma-1}{\gamma+1}\tau_Z, \tau_Z]$ since the Hugoniot has an asymptote at $\tau = \frac{\gamma-1}{\gamma+1}\tau_Z$.
    We first note the derivatives of $p^{\text{stiff}}_{\calH}(\tau)$:
    \begin{align}
        \frac{\diff}{\diff \tau} p_\calH^{\text{stiff}}(\tau)     & = -(p_L + p_{\infty,L}) \frac{\frac{4\gamma}{(\gamma + 1)^2}\tau_L}{\big( \frac{\gamma-1}{\gamma+1} \tau_L - \tau)^2}, \label{proof:deriv_hugo_p1} \\
        \frac{\diff^2}{\diff \tau^2} p_\calH^{\text{stiff}}(\tau) & = -(p_L + p_{\infty,L}) \frac{\frac{8\gamma}{(\gamma+1)^2}\tau_L}{\big( \frac{\gamma-1}{\gamma+1} \tau_L - \tau)^3}. \label{proof:deriv_hugo_p2}
    \end{align}
    Thus, the inequality to satisfy becomes,
    \begin{equation} \label{proof:energy_deriv_inequl}
        \begin{split}
            \frac{\diff^2}{\diff \tau^2} e^{\text{stiff}}_{\calH}(\tau) & = -\frac12 \frac{\diff^2}{\diff \tau^2}p^{\text{stiff}}_{\calH}(\tau)(\tau - \tau_Z) - \frac{\diff}{\diff \tau} p^{\text{stiff}}_{\calH}(\tau) \\
                                                                        & \geq \partial^2_{\tau} e^{\text{oracle}}_s(\tau) = \frac{K_s^{\text{oracle}}(\tau)}{\tau}.
        \end{split}
    \end{equation}
    Using \eqref{proof:deriv_hugo_p1} and \eqref{proof:deriv_hugo_p2},
    inequality \eqref{proof:energy_deriv_inequl} becomes,
    \begin{equation}
        -\frac{
            \frac{8\gamma \tau_Z^2 (p_Z + p_{\infty,Z})}{(\gamma + 1)^3}
        }{
            \big( \frac{\gamma-1}{\gamma+1}\tau_Z - \tau \big)^3
        }
        \geq \frac{K^{\text{oracle}}_s(\tau)}{\tau}.
    \end{equation}
    Recall that $K_Z = \gamma(p_Z + p_{\infty,Z})$, we then rewrite the expression above as,
    \begin{equation}
        K_Z \geq -\frac{K_s^{\text{oracle}}(\tau)}{\tau} \frac{(\gamma+1)^3}{8\tau_Z^2} \Big( \frac{\gamma-1}{\gamma+1} \tau_Z - \tau\Big)^3 =: f(\tau).
    \end{equation}
    Notice that $f(\tau_Z) = K_Z$ and $f\big(\frac{\gamma-1}{\gamma+1} \tau_Z \big) = 0$.
    Thus, if $f(\tau)$ is an increasing function on $\big(\frac{\gamma-1}{\gamma+1}\tau_Z, \tau_Z \big]$, then the result is proven.
    Differentiating we see that,
    \begin{equation*}
        f'(\tau) = \frac{(\gamma+1)^3}{8 \tau^2_Z} \Big( \partial^2_\tau p^{\text{oracle}}_s \Big( \frac{\gamma-1}{\gamma+1} \tau_Z - \tau \Big)^3 + 3\frac{K^{\text{oracle}}_s(\tau)}{\tau} \Big( \frac{\gamma-1}{\gamma+1} \tau_Z - \tau\Big)^2 \Big)
    \end{equation*}
    where we used the fact that $\partial_\tau (-K^{\text{oracle}}_s(\tau)/\tau) = \partial^2_\tau p_s^{\text{oracle}}(\tau)$.
    Noting the identity $\partial^2_\tau p(\tau,s) = 2KG/\tau^2$, we arrive at
    \begin{equation}
        f'(\tau) = \frac{(\gamma+1)^3}{8\tau_Z^2} \Big( \frac{\gamma-1}{\gamma+1} \tau_Z - \tau \Big)^2 \frac{K_s^{\text{oracle}}(\tau)}{\tau} \Big( \Big( \frac{\gamma-1}{\gamma+1} \tau_Z - \tau \Big) \frac{2G_s^{\text{oracle}}(\tau)}{\tau} + 3\Big).
    \end{equation}
    Since $\tau \in \big(\frac{\gamma-1}{\gamma+1} \tau_Z, \tau_Z\big]$, a sufficient condition for $f(\tau)$ to be an increasing function is,
    \begin{equation}
        2\frac{\gamma-1}{\gamma+1} \frac{\tau_Z}{\tau} \essinf_{\tau\leq\tau_Z} G^{\text{oracle}}_s(\tau) + 3 \geq 0
        \quad \Leftrightarrow \quad
        \essinf_{\tau \leq \tau_Z} G^{\text{oracle}}_s(\tau) \geq \frac{-3/2}{\frac{\gamma-1}{\gamma+1} \frac{\tau_Z}{\tau}}.
    \end{equation}
    From assumption \eqref{eq:shock_G}, the inequality holds for any $\gamma > 1$ and $\tau \in (\frac{\gamma-1}{\gamma+1} \tau_z, \tau_Z]$.
\end{proof}

\begin{lemma}[Minimum entropy principle with the average state] \label{lem:invariant_domain}
    Let $s_{\min} \eqq \min(s_L, s_R)$.
    Assume all conditions of Assumption~\ref{def:thermo_consistency} hold for all $s \geq s_{\min}$.
    Let $\gamma$, $\calp_{\infty,Z}$, and $\calq_Z$ defined in \eqref{eq:gammaz}--\eqref{eq:qz}, respectively for $Z \in \{L, R\}$.
    Then there exists $\epsilon_0 > 0$ such that for all $\epsilon \in (0,\epsilon_0)$, the minimum entropy principle is preserved, $s(1 / \overline{\rho}_{LR}, \sfe(\overline{\bu}_{LR})) \geq \min(s_L, s_R)$, where $\overline{\bu}_{LR} = \overline{\bu}_{LR}(\htlambda_{LR}; \tbu_L, \tbu_R, \bn)$ is defined by \eqref{eq:auxiliary_state} where $\htlambda_{LR}$ is any upper bound to the maximum wave speed of the extended Riemann problem (for the chosen $\epsilon > 0$).
\end{lemma}
\begin{proof}
    Recall that $\overline{\tbu}_{LR} = (\overline{\bu}_{LR}, \overline{\calP}_{\infty, LR}, \overline{\calQ}_{LR})^{\mathsf{T}}$ (equation \eqref{eq:extended_bar_states}).
    We define
    \begin{equation}
        \widetilde{\Psi}^s(\widetilde{\bsfU}) := \rho (\sfs^{\text{oracle}}(\bsfU) - s_{\min}).
    \end{equation}
    The result is proven if we can show that $\tPsi^s(\overline{\tbu}_{LR}) = \overline{\rho}_{LR} (\sfs^{\text{oracle}}(\overline{\bu}_{LR}) - s_{\min}) \geq 0$; that is, the state $\overline{\bu}_{LR}$ exists on or above the local minimum isentrope.
    In other words, $\sfe(\overline{\bu}_{LR}) = e^{\text{oracle}}\big(\frac{1}{\overline{\rho}_{LR}}, \sfs^{\text{oracle}}(\overline{\bu}_{LR})\big) \geq e^{\text{oracle}}(1/\overline{\rho}_{LR}, s_{\min})$.

    From~\citep{clayton2023robust}, the solution to the extended Riemann problem has positive density, hence $\overline{\rho}_{LR} > 0$. Therefore we only need to show that $\sfs^{\text{oracle}}(\overline{\bu}_{LR}) \geq s_{\min}$.
    Since $e^{\text{oracle}}(\tau, s)$ is convex, this implies that $\rho \sfs^{\text{oracle}}(\bu)$ is a strictly concave function (\citep[III, 1.1, Thm.~1.1]{godlewski_raviart_2021}), hence $\widetilde{\Psi}^s$ is strictly concave.
    Since $\htlambda_{LR}$ is an upper bound on the maximum wave speed in the extended Riemann problem, we have that,
    \[
        \overline{\tbu}_{LR} = \int^{1/2}_{-1/2} \tbu(x,t) \diff x,
    \]
    under the CFL constraint $t\htlambda_{LR} \leq \frac12$.
    Applying Jensen's inequality we have,
    \begin{equation*}
        \widetilde{\Psi}^s(\overline{\tbu}_{LR}) \geq \int_{-1/2}^{1/2} \widetilde{\Psi}^s(\tbu(x,t)) \diff x = \int_{-1/2}^{1/2} \rho(x,t) (\sfs^{\text{oracle}}(\bu(x,t)) - s_{\min}) \diff x.
    \end{equation*}
    We now proceed by showing that $s^{\text{oracle}}(\bu(x,t)) \geq s_{\min}$ for a.e.~$x$ in the solution to the extended Riemann problem.

    Let $v^*$ denote the contact velocity in the extended Riemann problem.
    Let $\epsilon_0 > 0$ be given from Lemma~\ref{lem:expansion_wave}, then for all $\epsilon \in (0,\epsilon_0)$, we have by Lemmas~\ref{lem:expansion_wave} and~\ref{lem:shock_wave}, that to the left of the contact, $\sfe(\bu(x,t)) \geq e_{s_L}^{\text{oracle}}(\tau(x,t))$ hence,
    \begin{equation*}
        \int_{-1/2}^{v^*t} \rho(\xi) s^{\text{oracle}}(\tau(\xi), \sfe(\bu(\xi)) \diff x \geq \int_{-1/2}^{v^*t} \rho(\xi) s^{\text{oracle}}(\tau(\xi), e^{\text{oracle}}_{s_L}(\tau(\xi)) \diff x,
    \end{equation*}
    using $\xi \eqq x/t$, where the inequality comes from the thermodynamic identity: $\big(\frac{\partial s}{\partial e}\big)_\tau = T > 0$.
    We have a similar result to the right of the contact:
    \begin{equation*}
        \int^{1/2}_{v^*t} \rho(\xi) s^{\text{oracle}}(\tau(\xi), \sfe(\bu(\xi)) \diff x
        \geq \int^{1/2}_{v^*t} \rho(\xi) s^{\text{oracle}}(\tau(\xi), e^{\text{oracle}}_{s_R}(\tau(\xi)) \diff x.
    \end{equation*}
    Combining these results, we see,
    \begin{equation}
        \tPsi^s(\overline{\tbu}_{LR})
        \geq \int_{-1/2}^{1/2} \rho(\xi) \big(s^{\text{oracle}}(\tau(\xi), e^{\text{oracle}}_{s_{\min}}(\tau(\xi))) - s_{\min}\big) \diff x.
    \end{equation}
    However, note that the right hand side is identically zero, since we are evaluating the specific entropy on the isentrope at $s = s_{\min}$.
    This completes the proof.
\end{proof}

\subsection{Robust wave speed upper estimate} \label{sec:robust_wave_speed_estimate}
Since Lemma~\ref{lem:invariant_domain} holds for all upper bounding wave speeds derived from $\epsilon > 0$ sufficiently small, we compute the wave speed as $\epsilon \to 0^+$.
Note for $\epsilon = 0$, the problem becomes becomes degenerate; however, the wave speed in the extended Riemann problem converges to a non-degenerate result as $\epsilon \to 0^+$.

The wave speed in the extended Riemann problem as $\epsilon \to 0^+$ is,
\begin{subequations} \label{eq:eps_zero_wave_speed}
    \begin{align}
        \tlambda_L(\calp^*) & = v_L - c_L \sqrt{\max\Big(\frac{\calp^* - p_L}{K_L}, 0 \Big) + 1} \\
        \tlambda_R(\calp^*) & = v_R + c_R \sqrt{\max\Big(\frac{\calp^* - p_R}{K_R}, 0 \Big) + 1}
    \end{align}
\end{subequations}
where $\calp^*$ is the root of the equations,
\begin{equation}
    \varphi(\calp) \eqq f_L(\calp) + f_R(\calp) + v_R - v_L = 0,
\end{equation}
with
\begin{equation}
    f_Z(\calp) \eqq \begin{cases}
        f_Z^{\text{exp}}(\calp) \eqq c_Z \log \Big( \frac{\calp + p_{\infty,Z}}{K_Z} \Big),     & \quad \text{ if } \calp \leq p_Z, \\
        f_Z^{\text{shock}}(\calp) \eqq \frac{\calp - p_Z}{\sqrt{\rho_Z(\calp + p_{\infty,Z})}}, & \quad \text{ if } \calp > p_Z,
    \end{cases}
\end{equation}
for $Z \in \{L, R\}$ and recall that $\calp_{\infty,Z} = K_Z - p_Z$.
Note that $\varphi(\calp) = 0$ can only be solved numerically.
We propose an alternative function to solve that provides an upper bound on the root of $\varphi(\calp) = 0$.

\subsubsection{Upper bound on $\calp^*$}
In order to avoid solving a nonlinear equation, we propose an upper bound similar to what is done in \citep[Sec.~5]{clayton_2022}.
That is, we shall compute $\wcalp^*$ such that $\wcalp^* \geq \calp^*$.
In this case we have that $\tlambda_L(\wcalp^*) \leq \tlambda_L(\calp^*) < \tlambda_R(\calp^*) \leq \tlambda_R(\wcalp^*)$.

Using Lemma~\ref{lem:f_z_inequal}, we see that the function, $f_L^{\text{exp}}(\calp) + f_R^{\text{exp}}(\calp) + v_R - v_L$ is a lower bound of $\varphi(\calp)$.
Unfortunately, this equation cannot be solved analytically.
Therefore, we use the further lower bound,
\begin{equation}
    \widehat{\varphi}_{RR}(\calp) \eqq c_L \log\Big( \frac{\calp + \calp_{\infty,\min}}{K_L} \Big) + c_R \log\Big( \frac{\calp + \calp_{\infty,\min}}{K_R} \Big) + v_R - v_L,
\end{equation}
where $\calp_{\infty,\min} = \min(\calp_{\infty,L}, \calp_{\infty,R})$.
Thus, the upper bound on the root of $\varphi(\calp) = 0$ is found by solving $\widehat{\varphi}(\calp) = 0$, which gives us,
\begin{equation} \label{eq:double_exp_approx}
    \wcalp^*_{RR} \eqq \exp\Big(\frac{c_L \log(K_L) + c_R \log(K_R) - (v_R - v_L)}{c_L + c_R}\Big) - \calp_{\infty,\min}.
\end{equation}

In the case that $\varphi(\calp_{\max}) \leq 0$, then the the solution consists of a double shock.
We can therefore define a lower bound to the double shock function:
\begin{equation}
    \widehat{\varphi}_{SS}(\calp) = \frac{\calp - p_L}{\sqrt{\rho_L(\calp + p_{\infty,\max})}} + \frac{\calp - p_R}{\sqrt{\rho_R(\calp + \calp_{\infty,\max})}} + v_R - v_L.
\end{equation}
The root of this equation is,
\begin{equation} \label{eq:double_shock_approx}
    \wcalp^*_{SS} = \Big(\frac{-b + \sqrt{b^2-4ac}}{2a}\Big)^2 - \calp_{\infty,\max}
\end{equation}
where $a = \sqrt{\rho_L} + \sqrt{\rho_R}$, $b = \sqrt{\rho_L \rho_R} (v_R - v_L)$, and $c = -\sqrt{\rho_R} (p_L + \calp_{\infty,\max}) - \sqrt{\rho_L}(p_R + \calp_{\infty,\max})$.

A visual depiction of the approximating functions is provided in Figure~\ref{fig:phi_func} for Riemann data which produces a shock-contact-shock wave structure.
The algorithm which determines the upper bound on $\calp^*$ is provided in Algorithm~\ref{alg:phat_star}.
Computing $\wcalp^*$ from this algorithm, we then compute the wave speeds by, $\lambda_L(\wcalp^*)$ and $\lambda_R(\wcalp^*)$.

\begin{algorithm}
    \caption{Upper bound on maximum wave speed computation}
    \label{alg:phat_star}
    \begin{algorithmic}
        \Procedure{WaveSpeedUpperBound}{$p_L, p_R, K_L, K_R$}
        \State $p_{\min} \gets \min(p_L, p_R)$
        \State $p_{\max} \gets \max(p_L, p_R)$
        \State $p_{\infty,\min} \gets \min(K_L - p_L, K_R - p_R)$
        \State $p_{\infty,\max} \gets \max(K_L - p_L, K_R - p_R)$
        \If{$\varphi(p_{\min}) \leq 0$}
        \State $\wp^* \gets -p_{\infty,\min}$ \Comment{Double expansion case}
        \ElsIf{$\varphi(p_{\max}) \geq 0$}
        \State $\wp^* \gets \min(p_{\max}, \wp^*_{RR})$ from \eqref{eq:double_exp_approx} \Comment{Expansion/shock case}
        \Else
        \State $\wp^* \gets \min(\wp^*_{RR},\wp^*_{SS})$ from \eqref{eq:double_exp_approx} and \eqref{eq:double_shock_approx} \Comment{Double shock case}
        \EndIf
        \State \textbf{return} $\htlambda_{LR} \gets \max(|\tlambda_L(\wp^*)|, |\tlambda_R(\wp^*)|)$ from \eqref{eq:eps_zero_wave_speed}
        \EndProcedure
    \end{algorithmic}
\end{algorithm}

\begin{figure} \label{fig:phi_func}
    \centering
    \begin{tikzpicture}
        \begin{axis}[
                width=10cm, height=7cm,
                xlabel={Pressure, \(p\)},
                ylabel={\(\varphi(p)\)},
                title={Comparison of approximating functions},
                legend pos=north west,
                legend cell align=left,
                grid=major,
            ]
            \addplot[domain=1:10,samples=200,thick,solid,black] {phiphi(x)};
            \addlegendentry{\(\varphi\)}
            \addplot[domain=1:10,samples=200,thick,solid,blue] {phiRR(x)};
            \addlegendentry{\(\widehat{\varphi}_{RR}\)}
            \addplot[domain=1:10,samples=200,thick,solid,red] {phiSS(x)};
            \addlegendentry{\(\widehat{\varphi}_{SS}\)}
        \end{axis}
    \end{tikzpicture}
    \caption{A comparison of the different approximating functions of $\varphi$ for some simple Riemann data. In this figure, the root of $\varphi(p) = 0$ occurs in the double shock case,
        whence we see that $\widehat{\varphi}_{SS}$ provides a very close approximation compared to $\widehat{\varphi}_{RR}$.}
\end{figure}

\begin{theorem}[Minimum entropy principle preserving] \label{thm:main_theorem}
    Let $\bsfU_i\upn\in\calA$ for all $i\in\calV$ and Assumption~\ref{def:thermo_consistency} holds for the equation of state.
    Then, the update $\bsfU_i\upnp$ given by~\eqref{eq:numerical_method} using the wave speed estimates $\{\htlambda_{ij}\}$ detailed in \S\ref{sec:robust_wave_speed_estimate}
    satisfies $\bsfU^{n+1}_i \in \calA_{s^i_{\min}}$ where $s^i_{\min} \eqq \min_{j\in\calI(i)} \sfs(\bsfU^n_j)$.
    That is, $\sfs(\bsfU_i\upnp) \geq s^i_{\min}$ for all $i \in \calV$.
\end{theorem}

\begin{proof}
    Define $\Psi^s_i(\bsfU_i^n) \eqq \rho^n_i (\sfs(\bsfU^n_i) - \min_{j\in\calI(i)} \sfs(\bsfU^n_j))$.
    Recall that $\rho\sfs$ is a strictly concave function since $e(\tau, s)$ is strictly convex (see \citep[III, 1.1, Thm.~1.1]{godlewski_raviart_2021}).
    Hence $\Psi^s_i$ is concave.
    Then,
    \begin{equation*}
        \Psi^s_i(\bsfU^{n+1}_i) \geq \sum_{k\in\calI(i)} \omega_{ik} \upn \Psi^s_i(\bsfU_k\upn) + \sum_{j\in\calI(i)\setminus\{i\}}\mu_{ij}\upn \Psi^s_i(\overline{\bsfU}^n_{ij}(\htlambda_{ij})).
    \end{equation*}
    Since $\Psi_i^s(\bsfU^n_k) \geq 0$ for all $k \in \calI(i)$ and by Lemma~\ref{lem:invariant_domain},
    \[
        \Psi^s_i(\overline{\bsfU}^n_{ij}(\htlambda_{ij})) \geq \overline{\rho}^n_{ij}(\min(\sfs(\bsfU^n_i), \sfs(\bsfU^n_j)) - \min_{l\in\calI(i)} \sfs(\bsfU^n_l)) \geq 0
    \]
    for all $j \in \calI(i) \setminus\{i\}$, hence $\Psi^s_i(\bsfU^{n+1}_i) \geq 0$.
    Thus the minimum principle on the specific entropy holds for the update $\bsfU_i\upnp$.
\end{proof}

\section{Equations of state} \label{sec:eos_details}
In this section, we discuss two practical equations of state that satisfy Assumption~\ref{def:thermo_consistency}.
\subsection{The simple MACAW equation of state}\label{sec:simple_macaw}
The simple MACAW equation of state, introduced in~\cite{lozanoSimpleMACAW}, is a Mie-Gr{\"u}neisen type equation of state and is a simplified version of the 
MACAW EOS described in~\citep{lozano2023analytic}.
This equation of state can be used for describing condensed materials such as copper.

The reference isentrope is chosen to be the cold curve -- which is the curve for which the isentrope and isotherm coincide at $s = T = 0$. 
The pressure is defined by:
\begin{equation}
    \pi(\tau, e) \eqq p(\tau, s(\tau, e)) = p_{\text{cold}}(\tau) + \frac{\Gamma_0}{\tau} (e - e_{\text{cold}}(\tau)),
\end{equation}
where $A > 0$ and $B > 0$ are material-dependent constants, $\tau_0$ is a reference specific volume and $\Gamma_0$ is a constant Gr{\"u}neissen parameter.
The cold curves for the specific internal energy and pressure are given by
\begin{subequations}
\begin{align}
    e_{\text{cold}}(\tau) &= A \tau_0 \Big( \Big( \frac{\tau}{\tau_0} \Big)^{-B} + B \Big( \frac{\tau}{\tau_0} \Big) - (B+1) \Big), \\
    p_{\text{cold}}(\tau) &= -e'_{\text{cold}}(\tau) = AB\Big(\Big( \frac{\tau_0}{\tau}\Big)^{B+1} - 1 \Big).
\end{align}
\end{subequations}
The isentropic bulk modulus is $\kappa(\tau, e) = K_{\text{cold}}(\tau) + \frac{\Gamma_0(\Gamma_0 + 1)}{\tau} (e - e_{\text{cold}}(\tau))$ where $K_{\text{cold}}(\tau) = -\tau p_{\text{cold}}'(\tau) = AB(B+1) \big( \frac{\tau}{\tau_0} \big)^{-(B+1)}$.
Following Proposition~\ref{prop:mie_gru_fund_deriv}, we have that the fundamental derivative for the simple MACAW EOS satisfies $\sfG \geq \frac12 ( 1 + \min(B+1, \Gamma_0 + 1)) > 1$.   
\subsection{The reactant Davis equation of state}\label{sec:davis}
The Davis equation of state is typically used for modeling high explosives for both the reactants and products. 
Depending on the application, the references curves and parameters can be very different. We refer the reader to~\citep{davis2000complete} for the reactant EOS and~\citep{davis1993equation} for the product EOS.
In this work, we focus only on the reactant version. 
Formally, the Davis EOS is given by
\begin{equation} \label{eq:davis_eos}
    e(\tau, s) \eqq e_s(\tau) + \frac{c^0_v T_s(\tau)}{1 + \alpha_{ST}} \Big( \Big( 1 + \frac{\alpha_{ST}}{c^0_v}(s - s_0) \Big)^{1 + \alpha_{ST}^{-1}} - 1 \Big),
\end{equation}
where $e_s(\tau)$ is some reference isentrope, $c_v^0$ is a reference specific heat capacity at constant volume, $\alpha_{ST}$, is another adjustable parameter, $s_0$ is a reference specific entropy. The function $T_s(\tau)$ is defined by
\begin{equation}
    T_s(\tau) = T_0 \begin{cases}
        \big( \frac{\tau}{\tau_0} \big)^{-\Gamma_0}, &\quad \text{ if } \tau > \tau_0, \\
        \big( \frac{\tau}{\tau_0} \big)^{-(\Gamma_0 + Z)} \exp\big(-Z \big(1 - \frac{\tau}{\tau_0} \big) \big), &\quad \text{ if } \tau \leq \tau_0,
    \end{cases}
\end{equation}
where $\Gamma_0$ is a reference Gr{\"u}neisen coefficient, $\tau_0$ a reference specific volume, and $Z$ a material dependent parameter.
As described in~\citep[Sec.~III.C]{velizhanin2023notes}, setting $s_0 = c_v^0 / \alpha_{ST}$ allows the reactant Davis EOS to satisfy the 3rd law of thermodynamics. 
This is necessary for our methodology since the case when $s < c_v^0 / \alpha_{ST}$ implies that $e(\tau, s)$ is no longer convex (which contradicts Assumption~\ref{def:thermo_consistency}).
Therefore, we restrict ourselves to thermodynamic states with specific entropy above $c_v^0 / \alpha_{ST}$.

The reactant Davis EOS can be written in the general Mie-Gr{\"u}neisen form
\begin{equation} \label{eq:davis_mie_gru}
    \pi(\tau, e) \eqq p(\tau, s(\tau, e)) = p_s(\tau) + \frac{\Gamma(\tau)}{\tau} (e - e_s(\tau)),
\end{equation}
where the reference isentropes for $e$ and $p$ and the Gr{\"u}neisen coefficient are given by
\begin{subequations}
\begin{align}
    e_s(\tau) &= e_0 + \frac{A^2}{16B^2} \begin{cases}
        \exp(4By) - 4By - 1, &\, \text{ if } \tau > \tau_0, \\
        \sum_{n=2}^4 \frac{(4By)^n}{n!} + C \frac{(4By)^5}{5!} + \frac{4By^3}{3(1 - y)^3}, &\, \text{ if } \tau \leq \tau_0,
    \end{cases} \label{eq:davis_eos_e_ref} \\
    p_s(\tau) &= \frac{A^2}{4B\tau_0} \begin{cases}
        \exp(4By) - 1, &\, \text{ if } \tau > \tau_0, \\
        \sum_{n=1}^3 \frac{(4By)^n}{n!} + C \frac{(4By)^4}{4!} + \frac{y^2}{(1 - y)^4}, &\, \text{ if } \tau \leq \tau_0,
    \end{cases} \label{eq:davis_eos_p_ref} \\
    \Gamma(\tau) &= \begin{cases}
        \Gamma_0, &\, \text{ if } \tau > \tau_0, \\
        \Gamma_0 + Zy, &\,  \text{ if } \tau \leq \tau_0,
    \end{cases} \label{eq:davis_eos_gruneisen}
\end{align}
\end{subequations}
where $y(\tau) = 1 - \frac{\tau}{\tau_0}$ and $A$, $B$, $C$, and $e_0$ are some material parameters.
The isentropic bulk modulus for the reactant Davis EOS is given by
\begin{equation}
    \kappa(\tau, e) = -\tau p_s'(\tau) + \Big( \frac{\Gamma(\tau)(\Gamma(\tau) + 1)}{\tau} - \Gamma'(\tau) \Big) (e - e_s(\tau)).
\end{equation}
Note, from Table~\ref{tab:davis-parameters}, we set $Z = 0$ wherein $\kappa(\tau, e)$ becomes continuous but the fundamental derivative becomes discontinuous.

We now discuss the validity of Proposition~\ref{prop:mie_gru_fund_deriv} for the reactant Davis EOS. 
Since we set $Z = 0$, $\Gamma(\tau) = \Gamma_0 > 0$ and therefore, we only need to certify that $\inf_{\tau} (-\tau K'_s(\tau) / K_s(\tau)) > 1$.
As this is a function of a single variable, we visually inspect it over a large domain, to verify it is above 1.
For completeness, we provide the following formulas:
\begin{equation}
    K_s(\tau) = \frac{A^2}{\tau_0} (1 - y) \begin{cases}
        \exp(4By), & \text{if } \tau > \tau_0, \\
        \frac12 (4By)^2 + 4By + 1 + C \frac16 (4By)^3 + \frac{y(y+1)}{2B(1 - y)^5}, & \text{if } \tau \leq \tau_0,
    \end{cases}
\end{equation}
and
\begin{equation}
    -\tau K'_s(\tau) = 
    \frac{A^2}{\tau_0} (1 - y)
    \begin{cases}
        \exp(4By) \big( 4B(1-y) - 1 \big), &\, \text{ if } \tau > \tau_0, \\
        \begin{aligned}
        &(4B)^2y(1 - \tfrac32 y) + 4B(1 - 2y) - 1 \\
        &+ C(4B)^3 y^2(\tfrac12 - \tfrac23 y) + \frac{2y^2 + 5y + 1}{2B(1 - y)^5}
        \end{aligned}, &\, \text{ if } \tau \leq \tau_0.
    \end{cases}
\end{equation}
An alternative approach to assessing $G > 1$ is to observe the sound speed on each isentrope to see if $c(v,s)$ is a decreasing function of $v$.
For example we plot $c(v,s)$ for the Davis EOS with parameters from Table~\ref{tab:davis-parameters} in Figure~\ref{fig:sound_speed_davis}.
This is certainly not a proof, but can offer insight when dealing with equations of state that are very complicated.

\begin{figure} \label{fig:sound_speed_davis}
\centering
\begin{tikzpicture}
    \begin{axis}[
      width=10cm, height=7cm,
      xlabel={Specific volume, \(v\)},
      ylabel={specific entropy, \(s\)},
      title={Sound speed manifold, \(c(v,s)\)},
      grid=major,
      colormap/hot2,
    ]
        \addplot3[surf, 
                  shader=faceted,
                  samples=40,
                  samples y=40,
                  domain=0.5:5, 
                  y domain=0.5:2] 
                  {isen_c(x,y)};
    \end{axis}
\end{tikzpicture}
\caption{Surface plot of the sound speed, $c$, for the Davis reactant EOS on $[0.5,5]\times[0.5,2]$}
\end{figure}
\section{Numerical Illustrations} \label{sec:numerical_results}
In this section, we illustrate the methodology proposed above.
We first verify the accuracy of the method with a smooth analytical solution to the Euler equations. Then, we perform a suite of 1D and 2D problems for each equation of state.
\subsection{Preliminaries}
The numerical tests are conducted using
the high performance code,~\texttt{ryujin}~\citep{ryujin-2021-1, ryujin-2021-3}. The code
uses continuous $\polQ_1$ finite elements, based on the scheme of~\citep{guermond2016}, on quadrangular meshes for the
spatial approximation and is built upon the \texttt{deal.II} finite element
library~\citep{dealII95}.
We note that the scheme used here is formally first-order accurate in space.
The wave speed estimate used in the algorithm are the ones detailed in Section~\ref{sec:robust_wave_speed_estimate} (see Algorithm~\ref{alg:phat_star}).
For all tests, the time-stepping is done with a three stage, third-order Runge-Kutta method which is made to be invariant-domain preserving following the techniques introduced in~\citep{Ern_Guermond_2022}. The time step size, $\Delta t$, is set via an explicit CFL condition and is computed during the first stage of each time step.
For all tests, we set the $\textup{CFL}$ number to be $0.9$. We use four types of boundary conditions depending on the application: \textup{(i)} Dirichlet; \textup{(ii)} slip; \textup{(iii)} do nothing.

For the sake of simplicity, we fix the parameters for each equation of state for all simulations. The Simple MACAW parameters are given in Table~\ref{tab:macaw-parameters} and are derived from~\cite[Tab.~1]{lozanoSimpleMACAW}.
The parameters for the reactant Davis equation of state are given in Table~\ref{tab:davis-parameters} and are derived from~\cite[Tab.~I]{jadrich2023uncertainty}.
\begin{table}[!ht]
    \centering
    \caption{Simple MACAW parameters}
    \label{tab:macaw-parameters}
    \begin{tabular}{lcccc}
        \toprule
        $\tau_0$ ($\SI{}{\frac{cm^3}{g}}$) & $\Gamma_0$ & $A$ (GPa) & $B$   \\
        \midrule
        $\frac{1}{8.952}$                  & $0.5$      & $7.3$     & $3.9$ \\
        \bottomrule
    \end{tabular}
\end{table}
\begin{table}[!ht]
    \centering
    \caption{Reactant Davis EOS parameters}
    \label{tab:davis-parameters}
    \begin{tabular}{lcccccccc}
        \toprule
        $\tau_0$ ($\SI{}{\frac{cm^3}{g}}$) & $\Gamma_0$ & $T_0$ ($\SI{}{K}$) & $A$ ($\SI{}{\frac{mm}{\mu s}}$) & $B$   & $C$   & $Z$ & $c^0_v$ ($\SI{}{\frac{kJ}{g \cdot K}}$) & $\alpha_{ST}$ \\
        \midrule
        $1/1.71$                           & 0.8437     & 298.15             & 2.434                           & 2.367 & 1.024 & 0   & 0.001072                                & 0.8484        \\
        \bottomrule
    \end{tabular}
\end{table}
The units for quantities of interest are as follows: $[\rho] = \SI{}{g~cm^{-3}},~[p] = \SI{}{GPa},~[e] = \SI{}{kJ~g^{-1}},~[T] = \SI{}{K},~[c] = \SI{}{mm~\mu s^{-1}}$.
From the discussion in Section~\ref{sec:eos_details}, in the case of the Simple MACAW EOS, we know that $G$ is bounded away from 1.
Hence the numerical method guarantees the minimum principle on the specific entropy.
Regarding the reactant Davis EOS, $G > 1$ has been observed upon visual inspection as described in Section~\ref{sec:davis}.

\subsection{Convergence test}\label{sec:convergence}
We now verify the accuracy of the methodology with a smooth, analytical solution to~\eqref{euler}.
We consider the one-dimensional test proposed in~\citep[Sec.~5.2]{Guermond_Nazarov_Popov_Tomas_SISC_2019} consisting of a smooth density wave with constant pressure. Furthermore, this test serves as a sanity check for the implementation of the equation of state of interest.
For the sake of brevity, we omit the details for the solution here and instead refer the reader to~\citep[Sec.~5.2]{Guermond_Nazarov_Popov_Tomas_SISC_2019} as well as~\citep[Sec.~5.3.1]{guermond2023}.

The tests are performed on uniform meshes with the domain $D = (0,\SI{1}{mm})$ until final time $t_{\text{final}} = \SI{0.2}{\mu s}$. The first mesh is composed of 100 cells. The other meshes are obtained by uniform refinement via bisection. We set Dirichlet boundary conditions for all tests. We report in Table~\ref{tab:smooth_wave} the respective convergence rates for each equation of state used with:
\begin{equation}
    \delta_1(t):=\frac{\|\rho_h(t)-\rho(t)\|_{1}}{\|\rho(t)\|_{1}}
    + \frac{\|\bbm_h(t)-\bbm(t)\|_{1}}{\|\bbm(t)\|_{1}} \\
    + \frac{\|E_h(t)-E(t)\|_{1}}{\|E(t)\|_{1}},
    \label{eq:delta}
\end{equation}
where $\rho(t),\,\bbm(t),\,E(t)$ are the exact states at time $t$, and
$\rho_h(t),\,\bbm_h(t),\,E_h(t)$ are the finite element approximations
at time $t$ for the respective conserved variables. We
observe that the convergence rate approaches 1 for each equation of state.

\begin{table}[ht]
    \centering
    \begin{tabular}[b]{lcrcr}
        \toprule
                  & \multicolumn{2}{c}{Simple MACAW} & \multicolumn{2}{c}{Davis}
        \\
        \cmidrule(lr){2-3} \cmidrule(lr){4-5}
        $|\calV|$ & $\delta_1(t_{\text{final}})$     &                           & $\delta_1(t_{\text{final}})$ &
        \\[0.3em]
        101       & \num{0.189314}                   & --                        & \num{0.209524}               & --   \\
        201       & \num{0.14934}                    & 0.34                      & \num{0.156638}               & 0.42 \\
        401       & \num{0.107107}                   & 0.48                      & \num{0.107045}               & 0.55 \\
        801       & \num{0.069566}                   & 0.62                      & \num{0.0670804}              & 0.67 \\
        1601      & \num{0.0412965}                  & 0.75                      & \num{0.0391279}              & 0.78 \\
        3201      & \num{0.0229007}                  & 0.85                      & \num{0.0216512}              & 0.85 \\
        6401      & \num{0.0121322}                  & 0.92                      & \num{0.0115331}              & 0.91 \\
        12801     & \num{0.00625326}                 & 0.96                      & \num{0.00598355}             & 0.95 \\
        25601     & \num{0.00317535}                 & 0.98                      & \num{0.0030532}              & 0.97 \\
        \bottomrule
    \end{tabular}
    \caption{%
        The consolidated error~\eqref{eq:delta} and corresponding convergence
        rates for the one-dimensional smooth traveling wave problem under uniform refinement.}
    \label{tab:smooth_wave}
\end{table}

\subsection{1D pull apart -- Simple MACAW} \label{sec:macaw-pull-1d}
We now perform a 1D ``pull apart'' Riemann problem using the Simple MACAW equation of state. The solution profile consists of two expansion waves into vacuum. The initial left/right data is as follows:  $(\rho_L, v_L, p_L) = (8.93,  -1.76, 0)$ and $(\rho_R, v_R, p_R) = (8.93, 1.76, 0)$.
The computational domain is set to $D = (0, \SI{1}{mm})$ with Dirichlet boundary conditions. The final time of the simulation is set to $t_{\text{final}} = \num{5e-2}{\mu s}$.
The computations are performed on a sequence of meshes composed of $\SI[scientific-notation = false]{1000}{},\,\SI[scientific-notation = false]{5000}{},\,\SI[scientific-notation = false]{10000}{},\,\SI[scientific-notation = false]{100000}{}$ elements, respectively.
The results are shown in Figure~\ref{fig:simple-pull} for pressure (left) and density (right) profiles, respectively.
We note that we have made no assumptions on the positivity of pressure -- differing from the authors' original approach in~\citep{clayton_2022,clayton2023robust}.
We see that the pressure approaches its minimum value (which bounds all isentropes from below). For the simple MACAW equation of state, this value is $-A\times B\approx -\SI[scientific-notation=false]{28.47}{GPa}$.
One may also notice the aberrant ``spike'' that appears in the density profile.
The refinement shows that the width and length of this spike is diminishing and hence converging to the expected constant state at the end of the rarefactions.
\begin{figure}[!ht]
    \centering
    \includegraphics[width=0.49\linewidth, clip, trim={22, 0, 22, 2}]{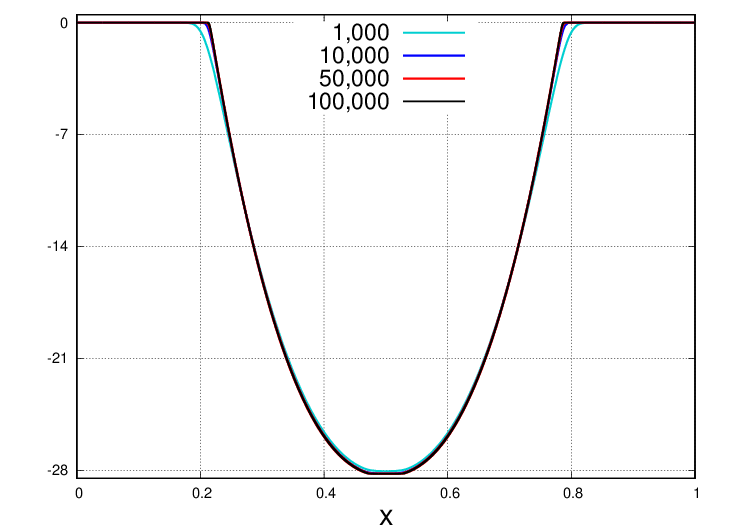}
    \includegraphics[width=0.49\linewidth, clip, trim={22, 0, 22, 2}]{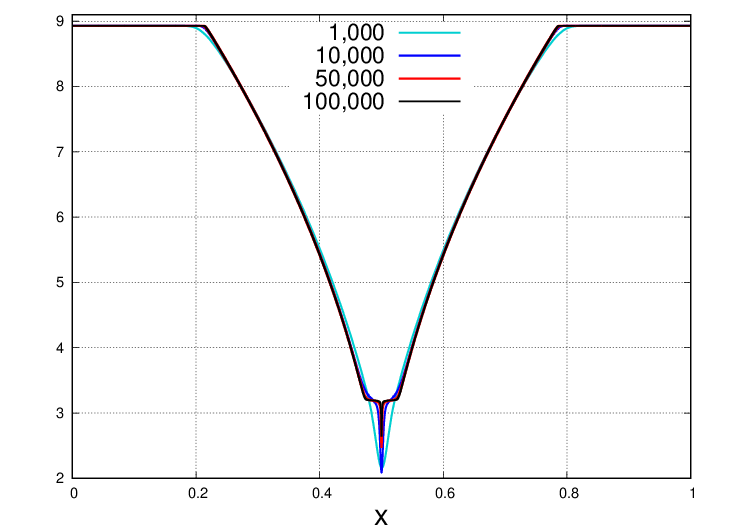}
    \caption{1D pull apart test at final time $t_{\text{final}} = \num{5e-2}{\mu s}$ for the Simple MACAW equation of state. The solution profiles shown here are pressure (left) and density (right).}
    \label{fig:simple-pull}
\end{figure}
\subsection{1D Leblanc-like shocktube -- Simple MACAW} \label{sec:macaw-RP}
We consider a 1D Leblanc-like shocktube problem for the Simple MACAW equation of state. The initial left/right data is as follows:  $(\rho_L, v_L, p_L) = (8.93,  0, 10^{8})$ and $(\rho_R, v_R, p_R) = (0.001, 0, -20)$. We recall that the minimum pressure on the cold curve is approximately $-\SI[scientific-notation=false]{28.47}{GPa}$ which is found by taking the limit as $\rho\rightarrow 0$ (\ie vacuum). Thus, the right state is close vacuum.

The computational domain for this test is $(-2, \SI{2}{mm})$ with Dirichlet boundary conditions. The final time of the simulation is set to $t_{\text{final}} = \num{1e-4}{\mu s}$.
The computations are performed on a sequence of meshes composed of $\SI[scientific-notation = false]{100000}{},\,\SI[scientific-notation = false]{200000}{},$ $\SI[scientific-notation = false]{400000}{},\,\SI[scientific-notation = false]{800000}{}$ elements, respectively.
The results are shown in Figure~\ref{fig:simple-leblanc} for pressure (left), density in low scale (middle), and a zoom in on the density (right) near the tail end of the expansion, respectively. We observe a small dip at the tail of the expansion wave which is similar behavior to what was observed in~\citep[Sec.~5.3.2]{guermond2016} for the standard Leblanc problem with ideal gas.
\begin{figure}[!ht]
    \centering
    \includegraphics[width=0.32\linewidth, clip, trim={22, 0, 15, 2}]{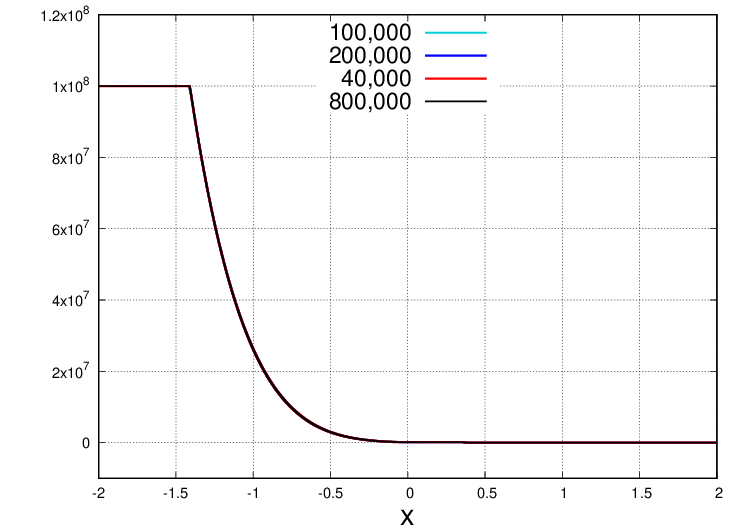}
    \includegraphics[width=0.32\linewidth, clip, trim={22, 0, 15, 2}]{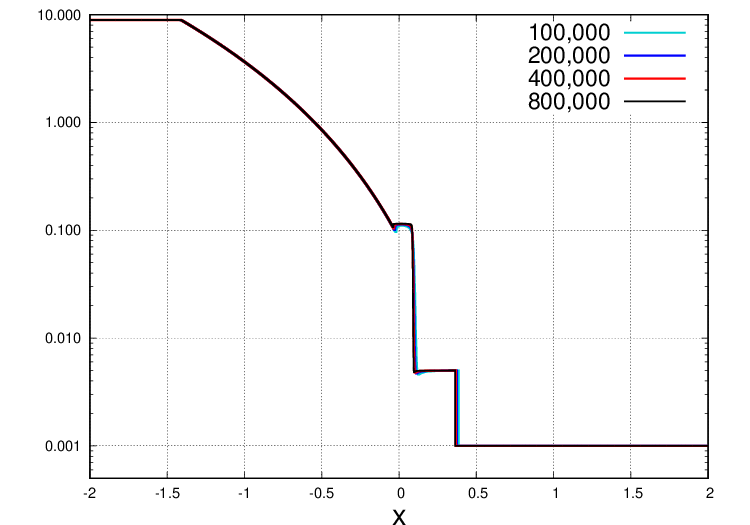}
    \includegraphics[width=0.32\linewidth, clip, trim={22, 0, 15, 2}]{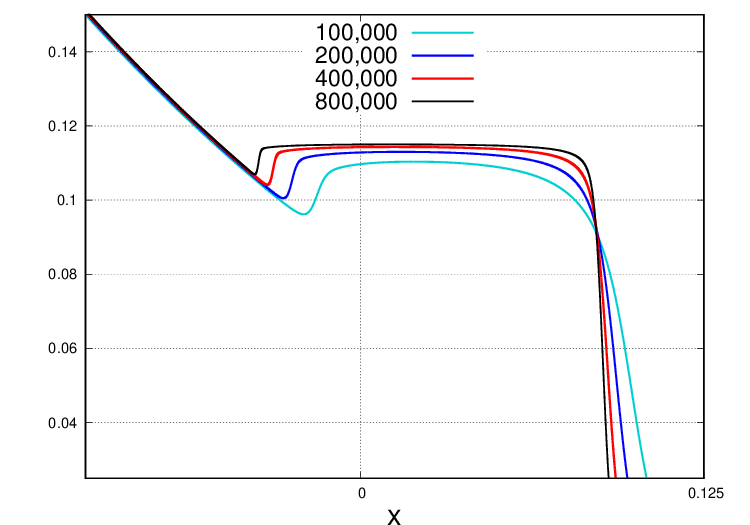}
    \caption{1D Leblanc-like shocktube at final time $t_{\text{final}} = \num{1e-4}{\mu s}$ for the Simple MACAW equation of state.}
    \label{fig:simple-leblanc}
\end{figure}

\subsection{1D blast wave -- Reactant Davis} \label{sec:davis-blast-1d}
We now reproduce the Woodward-Colella blast wave benchmark, first introduced in \citep[Sec.~IVa.]{woodward1984numerical}, using the reactant Davis EOS~\eqref{eq:davis_mie_gru} and a slightly modified initial state. The original benchmark sets a pressure value of $0.1$ in the ``low-pressure'' region $x\in (0.1, 0.9)$. We deviate from the stand value and instead set the pressure to $0$ in this region.
The computational domain is $D=(0,\SI{1}{mm})$ with slip boundary conditions. The initial state for the interacting blast wave problem is given by:
\begin{equation}
    \bu_0(x) = (\rho_0(x), v_0(x), p_0(x)) \eqq \begin{cases}
        (1, 0, 1000), & \, \text{ if } x \in [0,0.1],    \\
        (1, 0,   0),  & \, \text{ if } x \in (0.1, 0.9), \\
        (1, 0, 100),  & \, \text{ if } x \in [0.9,1].
    \end{cases}
\end{equation}
We set the simulation final time to $t_{\text{final}} = 0.0038$.
We show in~Figure~\ref{fig:davis-wc} the pressure and density profiles for different levels of mesh refinement. We see the solution is indeed converging and there were no issues with large pressure jumps.
\begin{figure}[!ht]
    \centering
    \includegraphics[width=0.49\linewidth, clip, trim={22, 0, 22, 2}]{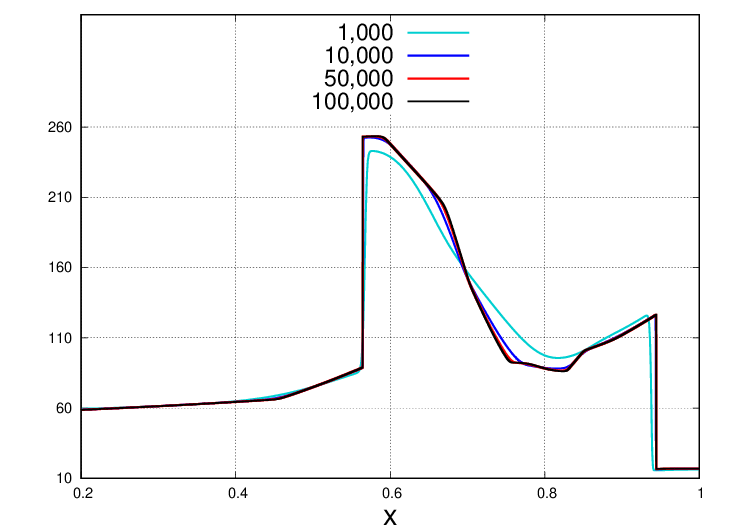}
    \includegraphics[width=0.49\linewidth, clip, trim={22, 0, 22, 2}]{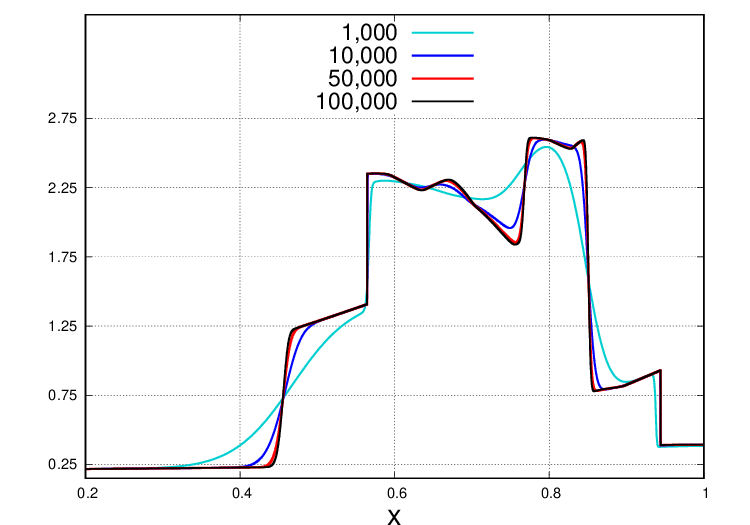}
    \caption{1D Woodward-Colella benchmark at final time $t_{\text{final}} = \num{0.038}{\mu s}$ with reactant Davis equation of state. The solution profiles shown here are pressure (left) and density (right).}
    \label{fig:davis-wc}
\end{figure}
\subsection{1D Entropy Test -- Reactant Davis}
We now illustrate numerically what it means to stay within admissible set~\eqref{eq:admissible} with a focus on the energy isentrope constraint.
For this test, we consider the Davis reactant equation of state. We recall that the energy isentrope for the Davis EOS is given by~\eqref{eq:davis_eos_e_ref}.
We consider a 1D Riemann problem where the left and right states have the same specific entropy so that they lie on the same isentrope:
\begin{equation} \label{eq:equal_entropy_states}
    (\rho_0(x), v_0(x), p_0(x)) \eqq \begin{cases}
        (2.5, -0.5, 9.980955089), & \text{ if } x \leq 5, \\
        (1, 0.5, -1.18049100646), & \text{ if } x > 5.
    \end{cases}
\end{equation}
The computational domain is given by $D = (0, 10~\SI{}{cm})$ with ``do nothing'' boundary conditions. We use a mesh composed of 12,800 elements.
We run until final time $t_{\text{final}}= 0.1$.
We plot the initial energy isentrope (black) as well as the initial specific internal energy (red circles) as a function of $\tau\eqq\frac{1}{\rho}$ in Figure~\ref{fig:davis-entropy} (left). We see that the initial energy data lies on the isentrope. We further plot the final energy state as a function of $\tau$ in the same figure (blue circles). An invariant-domain preserving method should keep the updated energy on or above the isentrope -- which we see in Figure~\ref{fig:davis-entropy}. For completeness, we also plot the initial and final energy as a function of the spatial coordinate in Figure~\ref{fig:davis-entropy} (right).

\begin{figure}
    \centering
    \includegraphics[width=0.495\linewidth,trim={12 0 12 0}, clip]{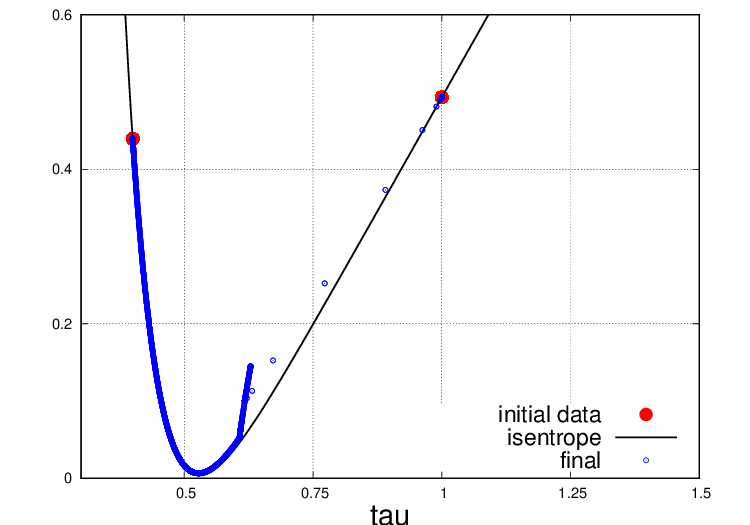}
    \includegraphics[width=0.495\linewidth,trim={12 0 12 0}, clip]{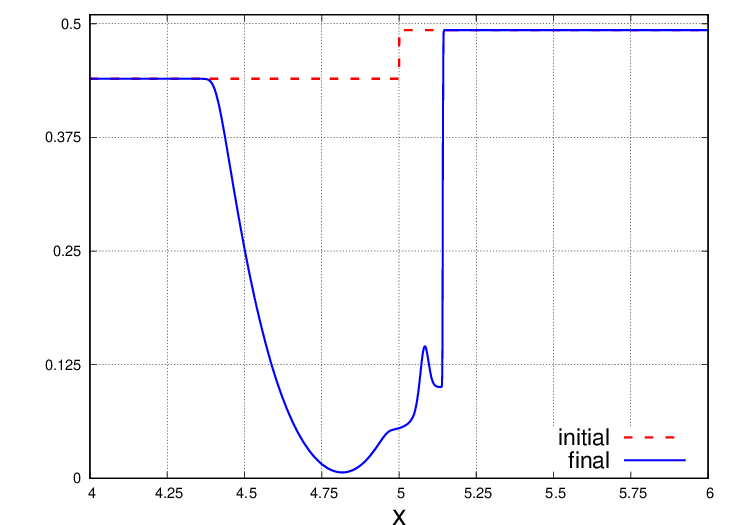}
    \caption{1D Riemann Problem - Davis EOS entropy test. The energy isentrope and initial/final energy states as a function of tau (left). The initial and final specific internal energy as a function of the spatial coordinate (right).}
    \label{fig:davis-entropy}
\end{figure}

\subsection{2D pull apart -- Simple MACAW} \label{sec:macaw-pull-2d}
We now perform a 2D extension of the 1D pull apart problem described in Section~\ref{sec:macaw-pull-1d}.
To stress the numerical method, we increase the velocity magnitude from $1.76~\SI{}{mm / \mu s}$ to $160~\SI{}{mm / \mu s}$ which leads to a Mach number of about $40.65$. The ``pulling apart'' is a consequence of defining the velocity vector to face in different directions in four separate quadrants (see: Figure~\ref{fig:macaw-2d-setup}). The set up for the problem is as follows. The initial density is to $8.93$ everywhere and the initial pressure is $0$.
The computational domain is defined to be $D = (0, \SI{2}{mm})\times(0, \SI{2}{mm})$ with ``do nothing'' boundary conditions.
Let quadrant I be the subdomain $D_{\textup{I}}=(1, \SI{2}{mm})\times(1, \SI{2}{mm})$; quadrant II the subdomain $D_{\textup{II}}=(0, \SI{1}{mm})\times(1, \SI{2}{mm})$; quadrant III the subdomain $D_{III}=(0, \SI{1}{mm})\times(0, \SI{1}{mm})$; quadrant IV the subdomain $D_{IV}=(1, \SI{2}{mm})\times(0, \SI{1}{mm})$. The velocity vector in each quadrant is defined as follows. Let $v_0\eqq160~\SI{}{mm / \micro s}$.
Then $\bv_I = (0, -v_0)^{\mathsf{T}}, \bv_{II} = (-v_0, 0)^{\mathsf{T}}, \bv_{III} = (0, -v_0)^{\mathsf{T}}, \bv_{IV} = (v_0, 0)^{\mathsf{T}}$. We set the final time to be $t_{\text{final}}=\SI{8e-3}{\micro s}$.
We perform the computations on a sequence of three meshes composed of:
$\SI[scientific-notation=false]{262144},\:\SI[scientific-notation=false]{4194304},\: \SI[scientific-notation=false]{67108864}{}$ elements to demonstrate convergence.
The numerical density profile for each mesh is shown in Figure~\ref{fig:macaw-2d}.
\begin{figure}[!ht]
    \centering
    \begin{minipage}{0.45\textwidth}
        \centering
        \includegraphics[width=\linewidth]{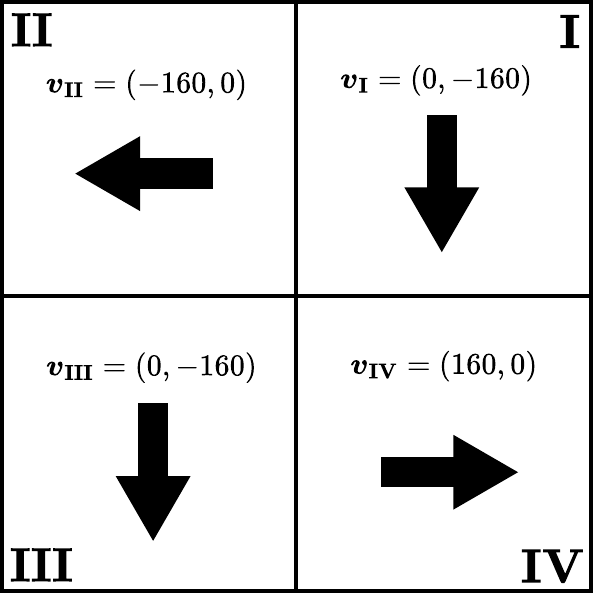}
        \caption{The initial setup for the 2D pull apart problem in Section~\ref{sec:macaw-pull-2d}.}
        \label{fig:macaw-2d-setup}
    \end{minipage} \hfill
\end{figure}

\begin{figure}[!ht]
    \centering
    \includegraphics[width=0.30\textwidth, clip, trim={0 0 0 0}]{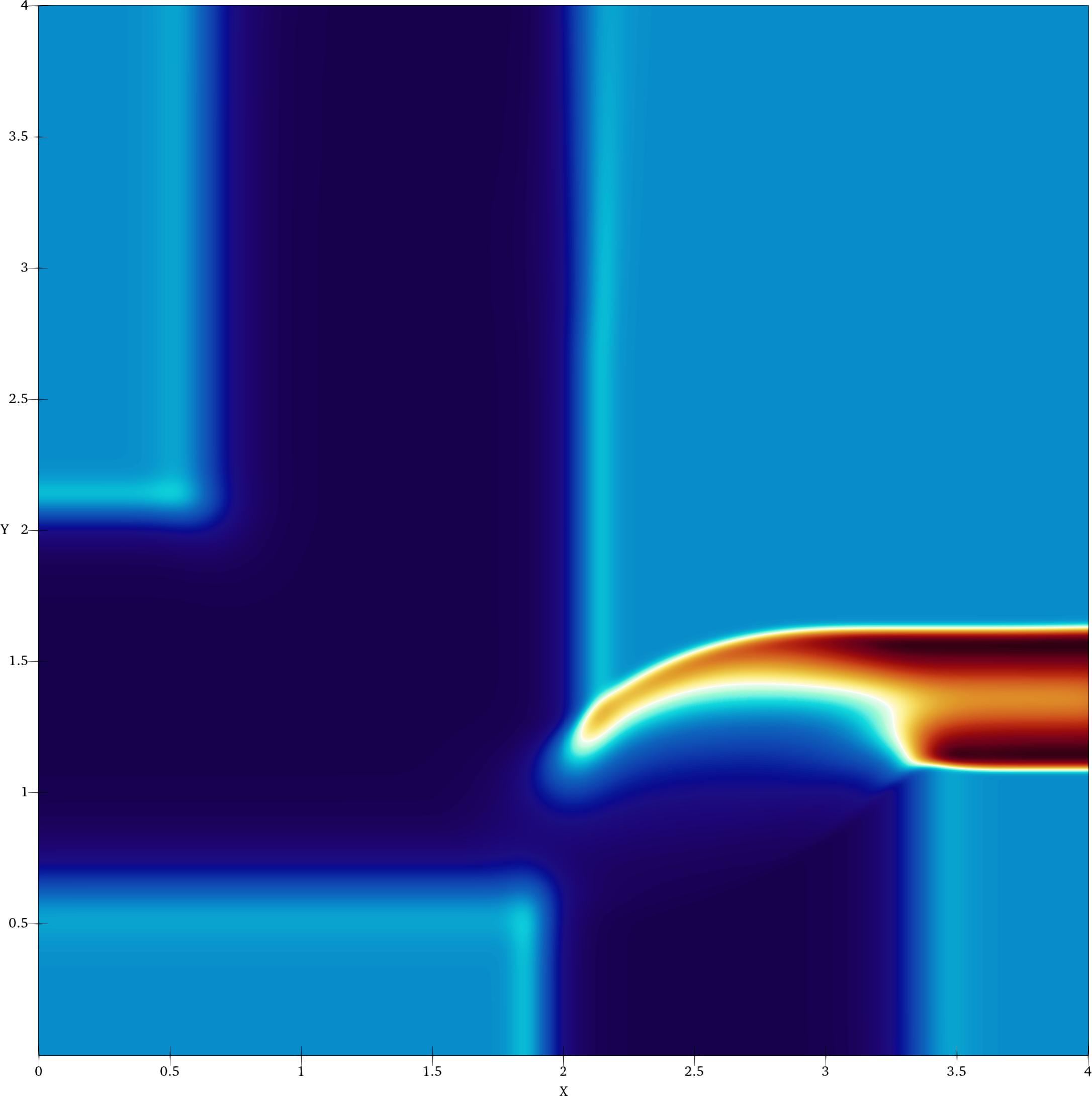}
    \includegraphics[width=0.30\textwidth, clip, trim={0 0 0 0}]{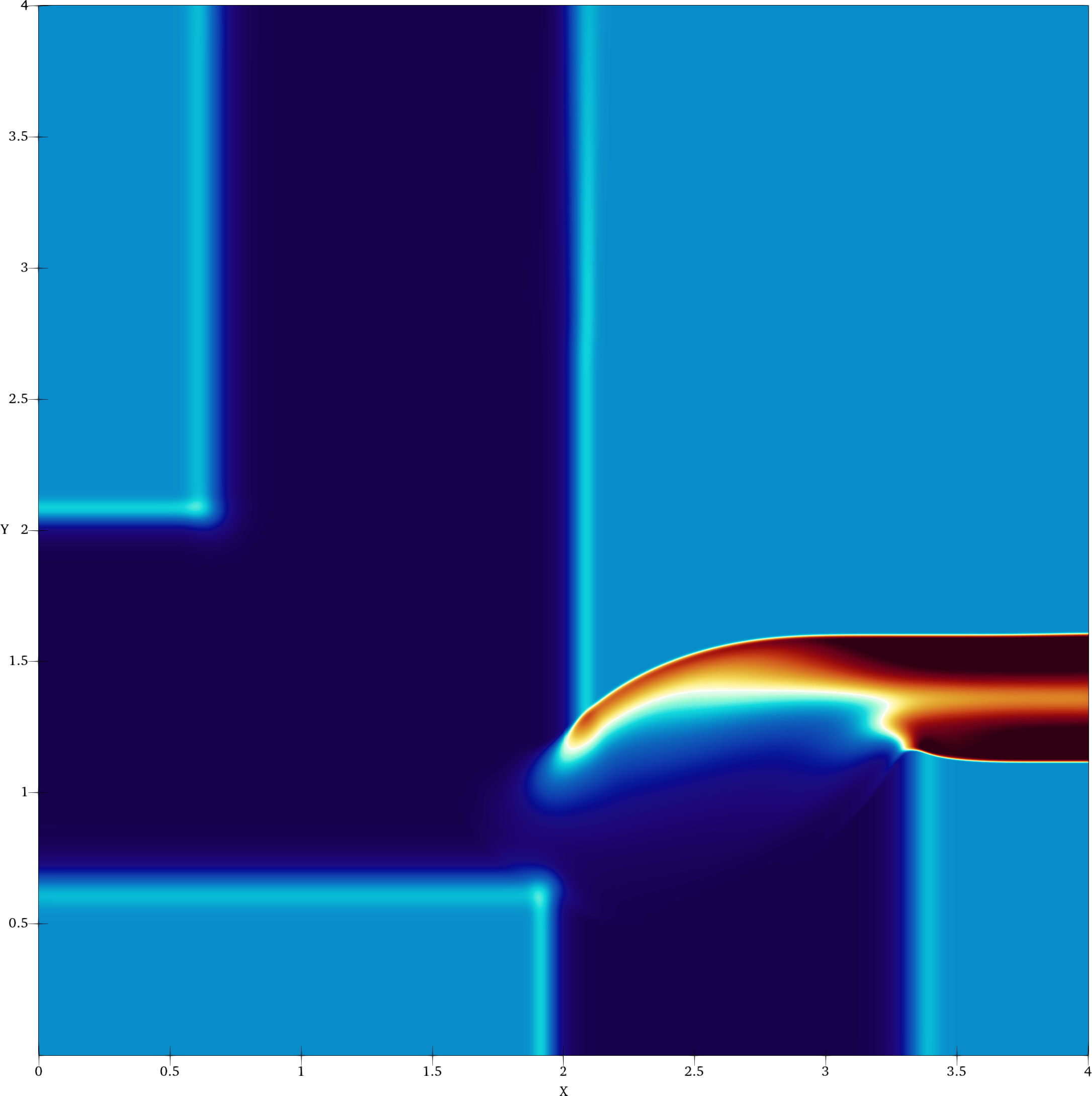}
    \includegraphics[width=0.35\textwidth, clip, trim={0 0 0 0}]{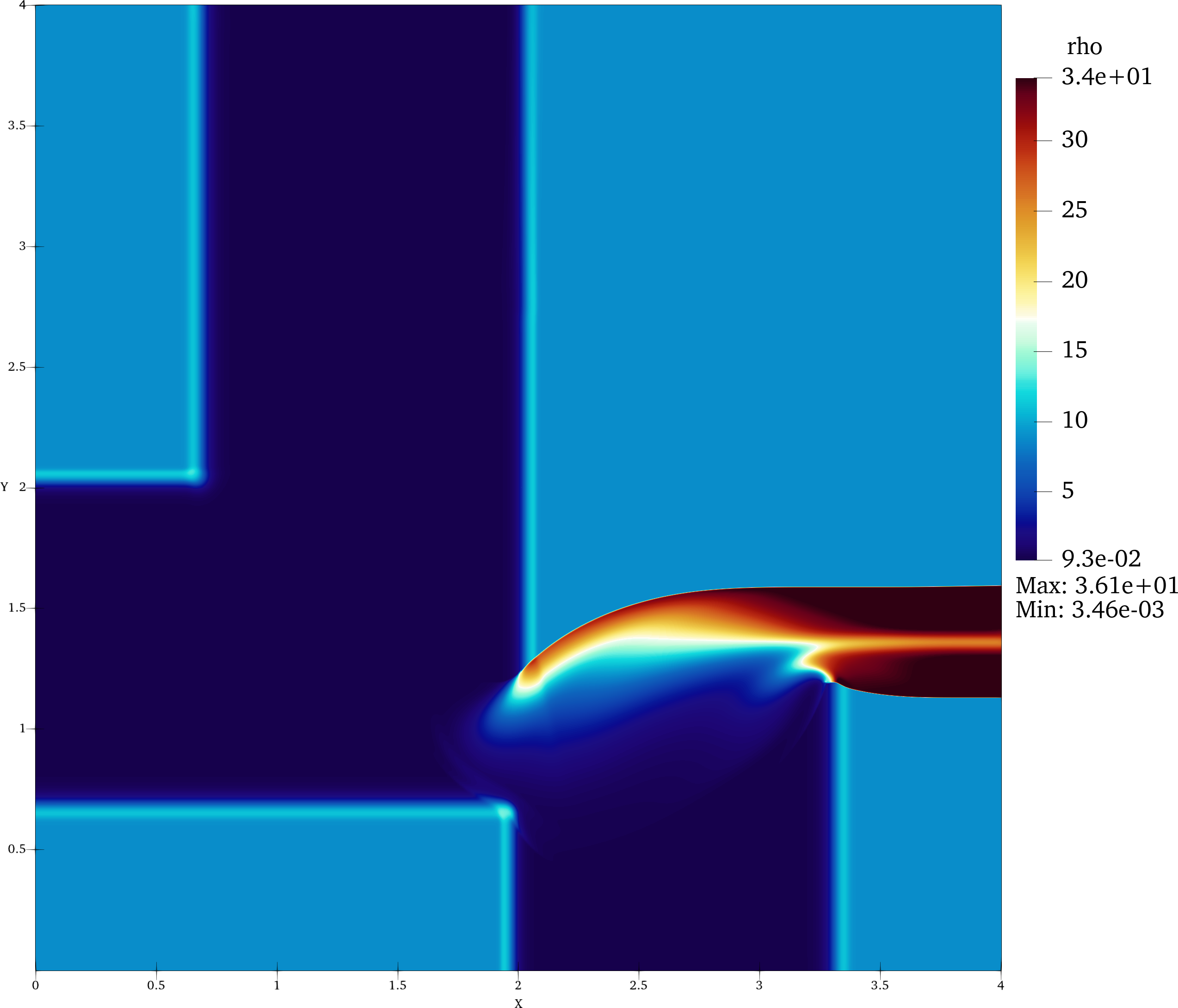}
    \caption{2D pull apart problem (density) at final time $t_{\text{final}}=\SI{8e-3}{\mu s}$ with the simple MACAW equation of state.}
    \label{fig:macaw-2d}
\end{figure}
\section{Conclusion}\label{sec:conclusion}
In this paper, we showed that any numerical method given in the form of~\eqref{eq:numerical_method} will satisfy the minimum principle on the specific entropy provided that: \textup{(i)} the equation of state satisfies Assumption~\ref{def:thermo_consistency}; \textup{(ii)}: one uses auxiliary states $\overline{\bsfU}^n_{LR}(\htlambda_{LR})$ defined by \eqref{eq:riemann_average} with the wave speed estimate, $\htlambda_{LR}$, provided in Algorithm~\ref{alg:phat_star}.
We numerically illustrated the robustness of the approach with a suite of test problems.
This work provides a robust foundation for extending the methodology to higher-order spatial approximation techniques.

\appendix
\section{Fundamental derivative for Mie-Gr{\"u}neisen} \label{app:fund_deriv}
Recall that Assumption~\ref{def:thermo_consistency} requires that $\essinf G(\tau, s) > 1$ for $\tau>0$ and $s \geq \sfs_0$ for some initial minimum specific entropy, $\sfs_0$.  
As the computation of the fundamental derivative can be tedious, we offer a convenient method for computing the essential infimum for a complete Mie-Gr{\"u}neisen type equation of state. 

We recall that pressure law for a Mie-Gr{\"u}neisen EOS is given by $\pi(\tau, e) \eqq p_{\text{ref}}(\tau) + \frac{\Gamma(\tau)}{\tau}(e - e_{\text{ref}}(\tau))$.
Here,  $\Gamma(\tau)$ is the Gr{\"u}neisen coefficient, and $p_{\text{ref}}$ and $e_{\text{ref}}$ are the reference pressure and specific internal energy curves, respectively.
It is common to define the references curves with a reference isotherm or isentrope.
In this work, we assume that the reference curves are only isentropes and use the following notation $e_s(\tau) \eqq e_{\text{ref}}(\tau)$ and $p_s(\tau) \eqq p_{\text{ref}}(\tau)$.
This is more natural since $p_s(\tau) = -e_s'(\tau)$.
This choice also yields simpler formulas for the isentropic bulk modulus and fundamental derivative.
\begin{lemma} \label{lem:mie_gru_identities}
    Let $\pi(\tau,e)$ be a Mie-Gr{\"u}neisen pressure law with reference isentropes, $e_s(\tau)$ and $p_s(\tau)$.
    Furthermore, assume sufficient differentiability of $\pi(\tau, e)$ and $\Gamma(\tau)$. Then, the isentropic bulk modulus and fundamental derivative are given by
    \begin{align}
        &\kappa(\tau, e) = K_s(\tau) + \big( \frac{\Gamma(\tau)(\Gamma(\tau) + 1)}{\tau} - \Gamma'(\tau) \big) (e - e_{s}(\tau)) \label{eq:mie_gru_bmod}, \\
        \sfG(\tau, e) &= \frac12 \Bigg[1 - \frac{
        \tau K_{s}'(\tau) - \Big( \tau \Gamma'' - 2 \Gamma' \Gamma + (\Gamma + 1) \big( \frac{\Gamma(\Gamma + 1)}{\tau} - \Gamma' \big) \Big)(e - e_{s}(\tau))
        }{
        K_{s}(\tau) + \big( \frac{\Gamma(\Gamma + 1)}{\tau} - \Gamma' \big) (e - e_s(\tau))
        } \Bigg], \label{eq:mie_gru_fundamental_deriv}
    \end{align}
    respectively, where $K_s(\tau) \eqq -\tau p_s'(\tau)$.
\end{lemma}
\begin{proof}
    Since $\pi(\tau, e(\tau,s)) = p(\tau, s)$, the derivatives as functions of $\tau$ and $e$ are determined by,
    \begin{align*}
        \kappa(\tau, e) &= K(\tau, s) = -\tau \partial_\tau p(\tau, s) = -\tau \partial_\tau \pi(\tau, e(\tau, s)) \\
        &= -\tau\big(\partial_\tau \pi(\tau, e) + \partial_e \pi(\tau, e) \partial_\tau e(\tau, s)\big) \\
        &= -\tau\big(\partial_\tau \pi(\tau, e) - \pi(\tau, e) \partial_e \pi(\tau, e)\big).
    \end{align*}
    Computing the derivative of $\pi$ and with respect to $\tau$ and $e$, one arrives at the result \eqref{eq:mie_gru_bmod}.
    Next, note that $\partial^2_\tau p(\tau, s) = \partial_\tau (-\kappa(\tau, e(\tau,s)) / \tau)$.
    Whence we have,
    \begin{equation*}
        \partial^2_\tau p(\tau, s) = - \frac{\tau \big(\partial_\tau \kappa(\tau, e) - p(\tau, s(\tau, e)) \partial_e \kappa(\tau, e)\big) - \kappa(\tau, e)}{\tau^2}.
    \end{equation*}
    Recall that $\sfG(\tau, e) \eqq G(\tau, s(\tau,e)) = -\frac12 \tau \frac{\partial^2_\tau p(\tau, s(\tau,e))}{\partial_\tau p(\tau, s(\tau,e))}$.
    Substituting into the definition of $\sfG$, we have,
    \begin{equation} \label{eq:G_of_tau_e}
        \sfG(\tau, e) = \frac12 \Big(1 - \frac{\tau}{\kappa(\tau, e)} \big( \partial_\tau \kappa(\tau, e) -\pi(\tau, e) \partial_e \kappa(\tau, e) \big) \Big).
    \end{equation}
    Similarly, computation of $\partial_\tau \kappa$ and $\partial_e \kappa$ provides us with formula \eqref{eq:mie_gru_fundamental_deriv}.
\end{proof}

\begin{proposition} \label{prop:mie_gru_fund_deriv}
    Assume the same conditions hold as in Lemma~\ref{lem:mie_gru_identities}. Additionally, assume: \textup{(i)} $\Gamma(\tau)$ is a positive decreasing and convex function; \textup{(ii)} $K_s(\tau)$ is a decreasing positive function; \textup{(iii)} $e \geq e_s(\tau)$.
    Then, the fundamental derivative is bounded below by
    \begin{equation}
        \sfG(\tau, e) \geq \frac12 \bigg[1 + \min\Big( \inf_{\tau > 0} \Big(\frac{-\tau K'_s(\tau)}{K_s(\tau)} \Big), \inf_\tau \Gamma(\tau) + 1 \Big) \bigg] =: g_{\min}.
    \end{equation}
\end{proposition}
\begin{proof}
    If the specific internal energy is on the reference isentrope; that is, $e = e_s(\tau)$, then $\sfG(\tau, e_s(\tau)) = \frac12\big(1 - \tau \frac{K'_s(\tau)}{K_s(\tau)}\big) \geq g_{\min}$.
    In the case when $e > e_s(\tau)$, from the assumptions on the Gr{\"u}neisen coefficient, we have that $\Gamma(\tau) > 0$, $\Gamma'(\tau) \leq 0$, and $\Gamma''(\tau) \geq 0$ for all $\tau > 0$.
    Therefore, a lower bound on \eqref{eq:mie_gru_fundamental_deriv} is given by
    \begin{equation}
        \sfG(\tau, e) \geq \frac12 \Bigg[1 + \frac{
        -\tau K_{s}'(\tau) + (\Gamma(\tau)+1)\big(\frac{(\Gamma(\tau) + 1)\Gamma(\tau)}{\tau} - \Gamma'(\tau) \big)(e - e_{s}(\tau))
        }{
        K_{s}(\tau) + \big(\frac{(\Gamma(\tau) + 1)\Gamma(\tau)}{\tau} - \Gamma'(\tau) \big)(e - e_s(\tau))
        } \Bigg].
    \end{equation}
    Using the general inequality, $\frac{a + b}{c + d} \geq \min\big(\frac{a}{c}, \frac{b}{d}\big)$ when $a, b, c, d > 0$, 
    we have that 
    \begin{equation}
        \sfG(\tau, e) \geq \frac12 \Bigg[1 + \min\Big(\frac{
        -\tau K_{s}'(\tau)}{K_s(\tau)}, \Gamma(\tau) + 1 \Big) \Bigg].
    \end{equation}
    Hence $\sfG(\tau, e) \geq g_{\min}$.
\end{proof}

\section{Lower bound estimate} \label{app:f_shock_bound}
We show that the shock curve bounds the expansion curve from above for $p > p_Z$.

\begin{lemma} \label{lem:f_z_inequal}
    For $p > p_Z$, we have that $f_Z^{\textup{shock}}(p) \geq f_Z^{\textup{exp}}(p)$.
\end{lemma}

\begin{proof}
    Since $f_Z^{\text{shock}}(p_Z) = f_Z^{\text{exp}}(p_Z)$, we only 
need to show that $\frac{d}{dp}f_Z^{\text{shock}}(p) \geq \frac{d}{dp}f_Z^{\text{exp}}(p)$ 
for all $p > p_Z$.
So we have,
\begin{align*}
    \frac{d}{dp}f_Z^{\text{shock}}(p) &= \frac{(p + p_{\infty,Z}) + (p_Z + p_{\infty,Z})}{2\sqrt{\rho_Z}(p + p_{\infty,Z})^{3/2}}, \\
    \frac{d}{dp}f_Z^{\text{exp}}(p) &= \frac{c_Z}{p + p_{\infty,Z}}.
\end{align*}
Letting $x \eqq p + p_{\infty,Z}$ and $a \eqq p_Z + p_{\infty,Z}$, the proof is then equivalent to showing that,
\begin{equation*}
    f(x) = \frac{x + a}{2\sqrt{ax}} \geq 1.
\end{equation*}
where we have used that $\sqrt{\rho_Z c_Z} = \sqrt{p_Z + p_{\infty,Z}}$.
It is seen that $f(x) = \frac{x+a}{2\sqrt{ax}}$ is an increasing function of $x$ (for $x \in [a, \infty)$) and $f(a) = 1$.
Hence the result is proven.
\end{proof}

\section*{Acknowledgments}
The authors thank Professors Jean-Luc Guermond and Bojan Popov at Texas A\&M University for their fruitful discussions and help provided throughout the project.
The authors also thank their colleagues from Los Alamos National Laboratory: Tariq Aslam, Eduardo Lozano, Joshua McConnell, and Clell Solomon for their insightful discussions on equations of state and continual support.
The authors further thank Tariq for the 2D set up outlined in Section~\ref{sec:macaw-pull-2d} of the manuscript.

\section*{Data availability}
Any code or data used for this article can be made available from the authors upon reasonable request.



\bibliographystyle{abbrvnat}
\bibliography{ref}

\end{document}